\documentclass[3p]{elsarticle}

\usepackage{lineno,hyperref}

\modulolinenumbers[1]

\journal{\textsc{Fractals}-Complex Geometry, Patterns and Scaling in Nature and Society}

\usepackage{amsfonts}
\usepackage{amsmath, amscd,amssymb,bm,amsbsy,epsf, amsfonts}
    \usepackage{enumerate}
    \usepackage{paralist}
\usepackage{bbm, dsfont}
\usepackage{graphicx,color}
\usepackage{subfigure}
\usepackage{mathrsfs}
\usepackage{verbatim}
\usepackage{inputenc}
\usepackage{bbm}
\usepackage{multicol}
\usepackage{balance}
\usepackage{verbatim}
\usepackage[table]{xcolor}
\usepackage{sfmath}
\usepackage{multirow}

\DeclareMathAlphabet{\mathpzc}{OT1}{pzc}{m}{it}

\DeclareMathOperator*{\supp}{supp}

\DeclareMathOperator*{\cl}{cl}

\DeclareMathOperator*{\card}{card}

\newtheorem{theorem}{Theorem}
\newtheorem{lemma}[theorem]{Lemma}
\newtheorem{corollary}[theorem]{Corollary}
\newtheorem{proposition}[theorem]{Proposition}
\newdefinition{definition}{Definition}
\newdefinition{hypothesis}{Hypothesis}
\newdefinition{remark}{Remark}
\newproof{proof}{Proof}
\newdefinition{example}[theorem]{Example}

\def\del{\boldsymbol{\partial } }
\def\defining{\overset{\mathbf{def} } =} 
\def\ind{\boldsymbol{\mathbbm{1} } }

\def\wconv{\overset{w}\rightharpoonup}

\def\K{\mathcal{K}}

\def\D{\mathcal{D} }
\def\In{\mathcal{I}_{n} }

\def\N{\boldsymbol{\mathbbm{N} } }
\def\R{\boldsymbol{\mathbbm{R} } }

\def\H{\mathcal{H}}
\def\nk{n_{k} }

\def\B{\mathfrak{B} }
\def\Bnot{\mathfrak{B}_{0} }
\def\Bn{\mathfrak{B}_{n} }

\def\Bnk{\mathfrak{B}_{n_{k} } }
\def\Bk{\mathfrak{B}_{k} }
\def\Bkl{\mathfrak{B}_{k -1} }

\def\b{\mathfrak{b} }
\def \bnot{\mathfrak{b}_{0} }
\def\Vb{V_{ \B}}

\def\pn{p_{n} }

\def\pnk{p_{n_{k}} }

\def\Tnot{T_{ \bnot}}
\def\Teps{(\Tnot\, \rho)_{ \epsilon}}
\def\Ueps{U_{ \epsilon}}

\begin{document}

\begin{frontmatter}

\title{A Discussion on the Transmission Conditions
for Saturated Fluid Flow Through
Porous Media With Fractal Microstructure}
\tnotetext[mytitlenote]{This material is based upon work supported by project HERMES 27798 from Universidad Nacional de Colombia,
Sede Medell\'in.}

\author{Fernando A Morales \sep $\&$ \sep Luis C Aristiz\'abal 
}
\address{Escuela de Matem\'aticas
Universidad Nacional de Colombia, Sede Medell\'in \\
Calle 59 A No 63-20 - Bloque 43, of 106,
Medell\'in - Colombia}

\author[mymainaddress]{Fernando A Morales}

\cortext[mycorrespondingauthor]{Corresponding Author}
\ead{famoralesj@unal.edu.co}


\begin{abstract}
The present work is aimed to find suitable exchange conditions for saturated fluid flow in a porous medium, when a fractal microstructure is embedded in the porous matrix. Two different deterministic models are introduced and rigorously analyzed. Also, numerical experiments for each of them are presented to verify the theoretically predicted behavior of the phenomenon and some probabilistic versions are explored numerically, to gain further insight on the phenomenon.
\end{abstract}


\end{frontmatter}


\section{Introduction}
%
%
An important topic of interest in the analysis of saturated flow through porous media, is modeling the phenomenon considering an embedded microstructure in the rock matrix. Some achievements in the preexisting literature address the case when the microstructure is periodic: see \cite{ShowalterMicro, ShowWalk90, ShowWalk91} for the analytical approach and \cite{ArbBrunson2007, ArbLehr2006} for the numerical point of view. From a different perspective \cite{Morales3} presents the homogenization analysis for a non-periodic fissured system where the geometry of the cracks' surface satisfies $C^{1}$-smoothness hypotheses. However, none of these mathematical analysis accomplishments, takes in consideration a fractal geometric structure embedded in the porous medium, which is an important case due to the remarkable evidence of such presence. In particular, in \cite{KatzThompson} the authors found that pore space and pore interface have fractal features. See \cite{ZhengShiChen} for a pore structure characterization, including random growth models.  See  \cite{FederJosang} for fractal geometry results in tracing experiments within a porous medium, including dispersion, fingering and percolation. Finally, \cite{ZhangLiTangGuo} discusses the use of fractal surfaces in modeling the storage phenomenon in gas reservoirs.  

The current paper concentrates on finding adequate fluid transmission conditions for flow in porous media with an embedded fractal microstructure, as well as the well-posedness of the corresponding weak variational formulations. The goal of the work is to ``blend" the modeling of porous media flow with the fractal roughness of the microstructure. It differs from the previously mentioned achievements since the geometric feature of periodicity is replaced by that of self-similarity and it explores numerically, the effect of some randomness in the fractal geometry. On the other hand, the present study has a very different approach to the analysis on fractals from the preexisting literature, given that the mainstream PDE analysis on fractals concentrates its efforts in solving strong forms on the fractal domain \cite{Strichartz3}, the analysis of the associated eigenvalues and eigenfunctions \cite{BarlowKigami}, or determining the adequate function spaces \cite{Strichartz1}. 

In order to gain understanding of the phenomenon's key features, the study is limited to the 1-D setting, defining the domain of analysis as $\Omega\defining(0,1)$. Additionally, an adjustment of the classic stationary diffusion problem \eqref{Eqn Darcy Classic} below will be used, in order to introduce the fluid exchange transmission conditions across the fractal interface

%
%
%
\begin{align}\label{Eqn Darcy Classic}
& -\del \big(K\, \del p\big) = F &
& \text{in }\, \Omega \, , &
& p(0) = 0 \, ,    &
& \del p (1) = 0\, .
\end{align}
%
%
%
Here, $p$ stands for the pressure, $K\del p$ indicates the flux according to Darcy's law and $K$ denotes de permeability, which will be set as $K \equiv 1$ throughout this work. In addition, Dirichlet and Neumann boundary conditions are adopted on the extremes of the interval. Classical notation and results on function spaces $L^{2}(\Omega), H^{1}(\Omega)$ are used and standard letters $p, q, r, u, v$ denote the functions on these spaces. The letters $\B$, $\b$ stand for the fractal microstructure and its elements respectively. In particular, the following classical Hilbert space will be frequently used
\begin{subequations}\label{Eqn Hilbert Space on Fractal}
\begin{equation}\label{Eqn Functions Hilbert Space on Fractal}
\ell^{\,2}(\B)\defining \Big\{g: \B\rightarrow \R\, \big\vert \sum\limits_{\b\, \in\, \B}\vert g(\b) \vert^{2} < +\infty\Big\}\, ,
\end{equation}
endowed with its natural inner product
\begin{equation}\label{Eqn Inner Product Hilbert Space on Fractal}
\big\langle g, h \big\rangle_{\ell^{2}(\B)} \defining \sum\limits_{\b\, \in\, \B} g(\b) \, h(\b).
\end{equation}
\end{subequations}
In the next section, the fractal interface geometry is introduced, together with the adequate mathematical setting, in order to include it successfully in the PDE model.  
\subsection{Geometric Setting}
%
%
Throughout this work we limit to a particular type of fractal microstructure, first we introduce a definition and a related result (see \cite{Falconer})
\begin{definition}[Iterated Function Systems]
Let $D\subseteq \R^{\! N}$ be a closed set 
\begin{enumerate}[(i)]
\item A function $S: D\rightarrow D$ is said to be a \textbf{contraction} if there exists a constant $c\in[0, 1)$ such that $\big\vert S(x) - S(y) \big\vert\leq c \, \big\vert x - y \big\vert$ for all $x, y\in D$. Additionally, we say that $S$ is a \textbf{similarity} if $\big\vert S(x) - S(y) \big\vert = c \, \big\vert x - y \big\vert$ for all $x, y\in D$ and $c$ is said to be its \textbf{ratio}.

\item A finite family of contractions, $\{S_{1}, S_{2}, \ldots, S_{L} \}$ on $D$ with $L\geq 2$ is said to be an \textbf{iterated function system} or \textbf{IFS}.

\item A non-empty compact subset $F\subseteq D$ is said to be an \textbf{attractor} of the IFS $\{S_{1}, S_{2}, \ldots, S_{L} \}$ if 
\begin{equation}\label{Def Attractor of an IFS}
F = \bigcup_{i \, = \, 1}^{L}S_{i}(F)\, .
\end{equation}
In particular, if every contraction of the IFS is a similarity then, the attractor $F$ is said to be a \textbf{strictly self-similar} set. 
\end{enumerate}
\end{definition}
\begin{theorem}
Consider the IFS given by the contractions $\{S_{1}, S_{2}, \ldots, S_{L} \}$ on $D\subseteq \R^{\! N}$ then, there is a unique attractor $F\subseteq D$ satisfying the identity \eqref{Def Attractor of an IFS}.
\end{theorem}
\begin{proof}
See \textsc{Theorem} 9.1 \cite{Falconer}.
\qed
\end{proof}
\begin{remark}\label{Rem The Attractor}
In this work, the attractor $F$ of an IFS will prove to be important in an indirect way: not for the definition of a microstructure, but for analyzing the nature of the attained conclusions.  
\end{remark}
It is a well-known fact that under certain conditions (see \cite{Falconer}, \textsc{Lemma} 9.2) a strictly self-similar set $F$ has both, Hausdorff and Box dimensions which are equal (see \cite{Falconer}, \textsc{Theorem} 9.3), namely $d$. Moreover, if $c_{i}$ is the ratio of the similarity $S_{i}$, then 
\begin{equation}\label{Eq Dimension of Strictly Self-Similar Sets}
\sum\limits_{i\, =\, 1}^{L} c_{i}^{d} = 1\, .
\end{equation}
From now on, we limit our attention to fractal structures satisfying strict self-similarity. 
Finally, we introduce the type of microstructure to be studied in throughout this work.
\begin{definition}[Fractal Microstructure]\label{Def Fractal Microstructure}
We say that a set $\B\subseteq [0,1]$ is a \textbf{fractal microstructure} if it is countable and there exists a sequence of finite subsets $\{\B_{n}: n \geq 0\}$, together with an iterated function system of similarities $\{S_{i}: 1\leq i\leq L \}$ on $[0,1]$, satisfying the following conditions
\begin{enumerate}[(i)]

\item The set $\B_{0}$ is finite and $\B_{n}$ is recursively defined by
\begin{align}\label{Eq Fractal Microstructure}
& \B_{n} = \bigcup\limits_{i\,=\, 1}^{L} S_{i}(\B_{n - 1})\, ,  &
& \text{for all }\; n\in \N\, .
\end{align}

\item The sequence of sets is monotonically increasing and
%
%
\begin{equation}\label{Eq sigma development}
\Bn\underset{n}\uparrow \B\, .
\end{equation}
%
%

\end{enumerate}
In the following, we refer to $\{\B_{n}: n\geq 0\}$ as a $\bm{\sigma}$\textbf{-finite development} of $\B$. 
\end{definition}
\begin{remark}\label{Rem Properties of the sigma finite development}
Let $\B\subseteq [0,1]$ be a fractal microstructure and let $\{S_{i}: 1\leq i\leq L \}$ be its corresponding system of similarities, notice the following
\begin{enumerate}[(i)]
\item There may exist more than one $\sigma$-finite development of $\B$.

\item Given a $\sigma$-finite development $\{\B_{n}: n \geq 0\}$ then, for each $n \in \N$, the following relationships of cardinality must hold
\begin{align}\label{Eq Cardinality Relationships}
& \card(\B_{n})  \leq  L\, \card(\B_{n-1}) \, , & 
& \card(\B_{n})  \leq  L^{n}\, \card(\B_{0}) \,   , & 
& \card(\B_{n} - \B_{n -1}) \leq  L^{n -1}(L -1) \, \card(\B_{0})  \,  .
\end{align}

\item The fractal microstructure $\B$ is necessarily contained in the unique fractal attractor $F$ of the IFS of similarities $\{S_{i}: 1\leq i\leq L \}$.
\end{enumerate}
\end{remark}
%
%
%
%
\section{Unscaled Storage Model for the Interface Microstructure}\label{Sec A sigma finite Interface Microstructure}
Let $\B$ be a microstructure set, let $\{\Bn:n\geq 0 \}$ be a $\sigma$-finite development and consider the following sequence of strong interface problems
\begin{subequations}\label{Pblm strong problem n-stage}
   \begin{equation}\label{Eq strong n-stage equation}   
      - \del^{2} \pn= F \quad \text{in}\; \left[\b_{k-1}, \b_{k}\right] .
    \end{equation}
    With interface conditions
    \begin{equation}\label{Eq strong n-stage interface}      
    \begin{split}
    \pn \left(\b_{k}^{-}\right) & = \pn\left(\b_{k}^{+}\right)  , 
    \\
    \del \pn \left(\b_{k}^{-}\right)  -  \del \pn \left(\b_{k}^{+}\right) + 
   \beta\,  \pn \left(\b_{k}\right) & =  f\left(\b_{k}\right) \, , 
    \quad \forall \; 1\leq k\leq K_{n} - 1 \, 
    \end{split}
    \end{equation}
    and boundary conditions
    \begin{align}\label{Eq strong n-stage boundary}  
      & \pn (0)= 0 \, , & 
      & \del \pn(1)= 0 \,.
    \end{align}     
\end{subequations}
Here, $0 = \b_{0} < \b_{1} < \b_{2}< \ldots < \b_{\scriptscriptstyle K_{n}} = 1$ is a monotone ordering of $\Bn$ i.e., $K_{n} = \card(\Bn)$. The forcing term $F$ belongs to $L^{2}(0,1)$, $\beta > 0$ is a storage fluid exchange coefficient, and $f$ is a source on the interface. In order to attain the variational formulation of the problems above we define the following function space
\begin{definition}\label{Def function space setting}
Define the space
\begin{equation}\label{Def space of pressures direct}
V \defining \big\{u\in H^{\,1}(0,1): u(0) = 0 \big\} ,
\end{equation}
endowed with the inner product $\langle \cdot, \cdot\rangle_{\scriptscriptstyle V}: V\times V\rightarrow \R$
\begin{equation}\label{Def norm space of pressures direct}
\langle u, v\rangle_{\scriptscriptstyle V}\defining \int_{0}^{1} \del u \, \del v \, ,
\end{equation}
and the norm $\Vert u\Vert_{\scriptscriptstyle V} \defining \sqrt{\langle u, u\rangle_{\scriptscriptstyle V}} $.
\end{definition}
The variational formulation of \textsc{Problem} \eqref{Pblm strong problem n-stage} above is given by 
\begin{align}\label{Pblm Variational n stage Concise}
& \pn \in V : & 
& \int_{0}^{1}\del \pn \, \del q 
+ \beta \sum_{\b \, \in  \, \Bn } \pn  (\b) \,  q(\b)
= \int_{0}^{1}F\,q 
+ \sum_{\b \, \in \, \Bn} f(\b) \,  q(\b)\, , &
& \forall \, q\in V.
\end{align}
\begin{theorem}\label{Th boundedness of the sequence of solutions}
The problem \eqref{Pblm Variational n stage Concise} is well-posed. Moreover if $f\in \ell^{2}(\B)$ the sequence of solutions $\{\pn: n\in \N\}$ is bounded in $V$ and for each $n\in \N$ it holds that
\begin{equation}\label{Ineq a-priori estimate pressures}
\begin{split}
\Vert \pn\Vert_{V},\, \Big(\sum_{\b \, \in \, \Bn} \pn^{2}(\b)\Big)^{ 1/2 } 
\leq 
%
\frac{1}{\min\{1, \beta\} }
\Big(\Vert F \Vert_{L^{2} (0,1)}^{2} +\Vert f\Vert^{2}_{\ell^{2}(\B)}\Big)^{ 1/2 } .
\end{split}
\end{equation}
\end{theorem}
\begin{proof}
Clearly, the bilinear form $(q, r)\mapsto \int_{0}^{1}\del q \, \del r 
+ \beta \sum\limits_{\b \, \in  \, \Bn } q (\b) \,  r(\b)$ is continuous and $V$-elliptic. In addition, it is direct to see that $q\mapsto \int_{0}^{1}F\,q + \sum\limits_{\b \, \in \, \Bn} f(\b) \,  q(\b)$ is linear and continuous. Consequently, the well-posedness of \textsc{Problem} \eqref{Pblm Variational n stage Concise} follows from the Lax-Milgram Theorem (see  \cite{Showalter77}). For the boundedness of $\{\pn: n\in \N \}$, we test \eqref{Pblm Variational n stage Concise} with $\pn$ and get
\begin{equation} \label{Ineq Cauchy - Schwartz in n-stage problem}
\begin{split}
\int_{0}^{1}(\del \pn)^{2} 
+ \beta \sum_{\b \, \in  \, \Bn }  \pn^{2}  (\b) &
= \int_{0}^{1}F\,\pn 
+ \sum_{\b \, \in \, \Bn} f(\b) \,  \pn(\b)  \\
& \leq \Vert F \Vert_{L^{2} (0,1)} \Vert \pn \Vert_{L^{2} (0,1)}
+ \bigg(\sum_{\b \, \in \, \Bn} f^{2}(\b)\bigg)^{1/2} \, 
\bigg(\sum_{\b \, \in \, \Bn} \pn^{2}(\b) \bigg)^{1/2}\\
& \leq
\bigg(\Vert F \Vert_{L^{2} (0,1)}^{2} + \sum_{\b \, \in \, \B} f^{2}(\b)\bigg)^{ 1/2 }
\bigg(\K^{2}_{(0,1)} \, \Vert \del \pn \Vert_{L^{2} (0,1)}^{2} + 
\sum_{\b \, \in \, \Bn} \pn^{2}(\b)\bigg)^{ 1/2 }.
\end{split}
\end{equation}
The second and third lines were obtained applying the Cauchy-Schwartz inequality in $L^{2}(0,1)$, $\R^{\card(\Bn) }$ and $\R^{2}$ respectively. Also, $\K_{(0,1)}$ is the Poincar\'e constant associated to the domain $(0,1)$; in this particular case $\K_{(0,1)}\leq \dfrac{1}{\sqrt{2}}$  (see \cite{Showalter77}). Thus, 
\begin{equation*} 
\begin{split}
\max \Big\{\Vert \pn\Vert_{V}, 
\bigg(\sum_{\b \, \in \, \Bn} \pn^{2}(\b)\bigg)^{ 1/2 } \Big\}
& \leq 
\Big(\int_{0}^{1}(\del \pn)^{2}  + \sum_{\b \, \in \, \Bn} \pn^{2}(\b)\Big)^{  1/2 } \\
& \leq \frac{1}{\min\{1, \beta\} }
\Big(\Vert F \Vert_{L^{2} (0,1)}^{2} +\Vert f\Vert^{2}_{\ell^{2}(\B)}\Big)^{ 1/2 },
\quad \forall\, n\in \N .
\end{split}
\end{equation*}
Hence, \textsc{Estimate} \eqref{Ineq a-priori estimate pressures} follows.
\qed
\end{proof}
In the theorem above, particularly due to the a-priori \textsc{Estimate} \eqref{Ineq a-priori estimate pressures}, we observe that the sequence of solutions $\{p_{n}: n\in \N\}\subseteq V$ contains extra information, which is given by the boundedness of the term $\sum\limits_{\b \, \in \, \Bn} \pn^{2}(\b)$. Implicitly, this fact gives the subspace of convergence.
\begin{definition}\label{Def Space of Convergence}
Define the \textbf{Fractal Interface Space}
\begin{equation}\label{Def fractal space pressures direct}
\Vb \defining \Big\{u\in H^{\,1}(0,1): u(0) = 0, \; 
\sum_{\b\,\in\, \B}\vert u(\b) \vert^{2} <\infty \Big\} ,
\end{equation}
endowed with the inner product $\langle \cdot, \cdot\rangle: \Vb\times\Vb\rightarrow \R$
\begin{equation}\label{Def norm fractal space of pressures direct}
\langle u, v\rangle_{\scriptscriptstyle \Vb}\defining \int_{0}^{1} \del u \, \del v 
+ \sum_{\b\,\in\, \B}u(\b)\, v(\b)
\end{equation}
and the norm $\Vert u\Vert_{\scriptscriptstyle \Vb}^{2} \defining \langle u, u\rangle_{\scriptscriptstyle \Vb} $.
\end{definition}
\begin{remark}\label{Rem the little ell 2 equivalence}
\begin{enumerate}[(i)]
\item In the following we refer to the trace operator $\gamma_{\scriptscriptstyle \B}: \Vb \rightarrow \R^{\!\B}$, $u\mapsto u\vert_{\B}$ as the \textbf{fractal trace operator} on $\B$, which is well-defined for functions $u\in V$.

\item
The space $\Vb$ is the subspace of functions $u \in V$ whose trace on the microstructure set, $u\vert_{\B}$, belongs to $\ell^{\, 2}(\B)$.

\item
The space $\Vb$ is a Hilbert space due to the completeness of $V$ with the norm $\Vert \cdot \Vert_{\scriptscriptstyle V} $ and the completeness of $\ell^{\, 2}(\B)$ with the norm $\big\{\sum\limits_{\b\,  \in \, \B} u^{2}(\b)\big\}^{ 1/2 }$.

\end{enumerate}

\end{remark}
\begin{theorem}\label{Th weak convergence of the whole sequence}
Let $\{\pn: n\in \N\}\subseteq V$ be the sequence of solutions to the family of \textsc{Problems} \eqref{Pblm Variational n stage Concise}. Then, if it converges weakly in $V$ to an element $\xi\in V$, the sequence of functionals
$\Lambda_{n}:\ell^{\, 2} (\B)\rightarrow \R$ defined by
\begin{equation}\label{Def partial development functionals}
\Lambda_{n}(u) \defining \sum_{\b \,\in \, \Bn} p_{n}(\b)\, u(\b)
= \sum_{\b \,\in \, \B} p_{n}(\b)\, u(\b) \, \ind_{\Bn}(\b)\,,\quad \forall \,  u\in \ell^{\, 2}(\B),
\end{equation}
converges weakly to $\xi\vert_{\B}$ in $(\ell^{2}(\B))'$. In particular, $\xi\in \Vb$ and 
\begin{equation}\label{Eq convergence of the fractal functional}
\sum_{\b \, \in \, \Bn}  \pn(\b) \, q(\b) \xrightarrow[n\rightarrow \infty]{}
\sum_{\b \, \in \, \B}  \xi(\b) \, q(\b) \, , \quad
\forall\, q\in \Vb.
\end{equation}
\end{theorem}
\begin{proof}
Clearly $\Lambda_{n}\in \ell^{\, 2} (\B) '$ for all $n\in \N$ and due to the Riesz Representation theorem, we have that
\begin{equation*} 
\Vert\Lambda_{n}\Vert = \Big(\sum_{\b \,\in \, \Bn} p_{n}^{2}(\b) \Big)^{1/2}
\leq \frac{M}{\min \{ 1, \beta \} }\quad \forall\, n\in \N ,
\end{equation*}
where the last inequality holds due to \textsc{Estimate} \eqref{Ineq a-priori estimate pressures}. Consequently, there must exist a subsequence $\{n_{k}: k\in \N \}$ and an element $\eta\in \ell^{\, 2}(\B)$ such that
\begin{equation*} 
\Lambda_{\nk}(u) =  \sum_{\b \,\in \, \Bnk} \pnk(\b)\, u(\b) \xrightarrow[k \, \rightarrow \, \infty]{} 
\sum_{\b \, \in \, \B} \eta (\b)\, u(\b)
\,,\quad \forall \,  u\in \ell^{ \, 2}(\B).
\end{equation*}
In particular, for any $\b_{0} \in \B$, let $K\in \N$ be such that $\b_{0}\in \Bnk$ for all $k>K$. Then, recalling that $\ind_{ \{\b_{0}\} } \in \ell^{\,2}(\B)$, the expression above yields
\begin{align*} 
& \pnk(\b_{0}) =  \sum_{\b \,\in \, \Bnk} \pnk(\b)\, \ind_{\{\b_{0}\}}(\b)  \xrightarrow[k \, > \, K]{} 
\sum_{\b \, \in \, \B} \eta (\b)\, \, \ind_{\{\b_{0}\}}(\b) = \eta (\b_{0})
\,, &
& \forall \; \b_{0}\in\B.
\end{align*}
On the other hand, since $\pn\wconv \xi$ this implies that $\pn(x_{0})\rightarrow \xi(x_{0})$ for all $x_{0} \in [0,1]$.  In particular $\lim\limits_{\, k \, \rightarrow \, \infty}\pnk(\b_{0})  = \xi(\b_{0})$ for all $\b_{0} \in \B$, consequently $\xi\vert_{\B} = \eta$ and  $\xi\in \Vb$. Moreover, since the above holds for any convergent subsequence of $\{\Lambda_{n}: n\in \N\}\subseteq (\ell^{\, 2}(\B))'$, it follows that the whole sequence is weakly convergent i.e.,
\begin{align*} 
& \Lambda_{n} \wconv \xi\vert_{\B}\, , &
&  \text{weakly in } (\ell^{\, 2}(\B))' .
\end{align*}
From here, the convergence statement \eqref{Eq convergence of the fractal functional} follows trivially.
\qed
\end{proof}
\begin{remark}\label{Rem partial interface development functionals}
It is important to observe that the weak convergence hypothesis for the functionals $\Lambda_{n}$ in $(\ell^{\, 2}(\B))'$, is a stronger condition than the statement \eqref{Eq convergence of the fractal functional}, since it can not be claimed that the fractal trace operator $\gamma_{\scriptscriptstyle \B}: \Vb \rightarrow \ell^{\, 2}(\B)$, $q\mapsto q\vert_{\B}$ is surjective. Moreover, it will be shown that in most of the interesting cases this operator is not surjective.  
\end{remark}
\subsection{The Limit Problem}\label{Sec the limit problem}
Consider the variational problem with microstructure interface 
\begin{align}\label{Pblm Variational Microstructure Concise}
& p \in \Vb : &
& \int_{0}^{1}\del p \, \del q 
+ \beta \sum_{\b \, \in  \, \B} p  (\b) \,  q(\b)= \int_{0}^{1}F\,q 
+ \sum_{\b \, \in \, \B} f(\b) \,  q(\b)\,, &
& \forall \, q\in \Vb.
\end{align}
Where, $F\in L^{2}(0,1)$ and $f\in \ell^{\, 2}(\B)$. We claim that the solution $p$ of the problem above is the weak limit of the sequence $\{\pn: n\in \N \}\subseteq V$ i.e., \textsc{Problem} \eqref{Pblm Variational Microstructure Concise} is the ``\textbf{limit}" of \textsc{Problems} \eqref{Pblm Variational n stage Concise}. 
\begin{theorem}\label{Th well-posedness fractal interface problem}
The problem \eqref{Pblm Variational Microstructure Concise} is well-posed.
\end{theorem} 
\begin{proof}
Consider the bilinear form 
\begin{equation}\label{Def Bilinear Form Dyadic}
a(q, r)\defining \int_{0}^{1}\del q \, \del r 
+ \beta \sum_{\b \, \in  \, \B} q  (\b) \,  r(\b).
\end{equation}
Using the Cauchy-Schwartz inequality in each summand of the right hand side in the expression above, we conclude the continuity of the bilinear form. 
On the other hand, it is direct to see that $\min\{1, \beta\}\Vert q \Vert_{\scriptscriptstyle \Vb}^{2}\leq \vert a(q,q)\vert$, which implies that the bilinear form is $\Vb$-elliptic. Applying the Lax-Milgram Theorem, the result follows, see \cite{Showalter77}. 
\qed 
\end{proof}
Now we are ready to prove the weak convergence of the whole sequence of solutions of \textsc{Problems} \eqref{Pblm Variational n stage Concise} to the solution $p$ of \textsc{Problem} \eqref{Pblm Variational Microstructure Concise}.
\begin{theorem}\label{Th weak convergence of the solutions in V}
Let $\{\pn: n\in \N\}\subseteq V$ be the sequence of solutions of \textsc{Problems} \eqref{Pblm Variational n stage Concise}, then it converges weakly in $V$ to the unique solution $p$ of \textsc{Problem} \eqref{Pblm Variational Microstructure Concise}. 
\end{theorem}
\begin{proof} 
Since $\{\pn: n\in \N\}$ is bounded in $V$ there must exist a weakly convergent subsequence $\{\pnk: k\in \N \}$ and a limit $\xi\in V$. Test \eqref{Pblm Variational n stage Concise} with $q\in \Vb\subseteq V$ arbitrary, this gives
\begin{equation*} 
\int_{0}^{1}\del \pnk  \del q  + 
 \beta \sum_{\b \, \in  \, \Bnk } \pnk  (\b) \,  q(\b)= \int_{0}^{1}F\,q 
+ \sum_{\b \, \in \, \Bnk} f(\b) \,  q(\b) \, .
\end{equation*}
Letting $k\rightarrow \infty$ in the expression above and, in view of \textsc{Theorem} \ref{Th weak convergence of the whole sequence}, we get
\begin{align*} 
& \int_{0}^{1}\del \xi \, \del q  + 
 \beta \sum_{\b \, \in  \, \B } \xi  (\b) \,  q(\b) = \int_{0}^{1}F\,q 
+ \sum_{\b \, \in \, \B} f(\b) \,  q(\b) \,, &
& \forall \, q\in \Vb .
\end{align*}
Additionally, \textsc{Theorem} \ref{Th weak convergence of the whole sequence} implies that $\xi\in \Vb$. Therefore $\xi$ is a solution to \textsc{Problem} \eqref{Pblm Variational Microstructure Concise}, which is unique due to \textsc{Theorem} \ref{Th well-posedness fractal interface problem}; therefore we conclude that $\xi \equiv p$. Since the above holds for any weakly convergent subsequence of $\{\pn: n \in \N \}$ it follows that the whole sequence must converge weakly to $p$.  
\qed
\end{proof}
\begin{lemma}\label{Th norms convergence}
Let $\{\pn: n\in \N \}\subseteq V$ be the sequence of solutions of \textsc{Problems} \eqref{Pblm Variational n stage Concise} then
\begin{subequations}\label{Stmt strong convergence}
\begin{equation} \label{Stmt strong convergence in V}
\Vert \pn-  p \Vert_{\scriptscriptstyle V} \xrightarrow[ n \, \rightarrow \, \infty ]{} 0 .
\end{equation}
\begin{equation} \label{Stmt strong convergence of interface fractal norm}
\left\Vert \gamma_{\scriptscriptstyle \B} (\pn) \ind_{\Bn} -  \gamma_{\scriptscriptstyle \B}(p)\,\ind_{\B}\right\Vert_{\scriptscriptstyle \ell^{\, 2}(\B)} \xrightarrow[n \, \rightarrow \, \infty]{} 0 .
\end{equation}
Where $\gamma_{\scriptscriptstyle \B}(q) \defining q\,\ind_{\scriptscriptstyle \B }$ is the fractal trace operator on $V$. 
\end{subequations}
\end{lemma}
\begin{proof}
We know that $\pn \wconv p$ weakly in $V$ and $\pn \ind_{\Bn}\wconv p\ind_{\B}$ weakly in $\ell^{\,2}(\B)$ from \textsc{Theorems} \ref{Th weak convergence of the solutions in V} and \ref{Th weak convergence of the whole sequence} respectively, therefore
\begin{align*}
\int_{0}^{1}\left\vert \del p\right\vert^{2}  \leq 
\liminf_{n}\int_{0}^{1} \left\vert \partial \pn\right\vert^{2} ,\\
\sum_{\b\,\in\, \B}p^{2}(\b) \leq  \liminf_{n} \sum_{\b\,\in\, \B}\pn^{2}(\b)\,\ind_{\Bn}(\b) = 
\liminf_{n} \sum_{\b\,\in\, \Bn}\pn^{2}(\b) . 
\end{align*}
On the other hand, testing \eqref{Pblm Variational Microstructure Concise} on the diagonal $\pn$, we get
\begin{equation*}
\int_{0}^{1}\del \pn^{2} 
+ \beta \sum_{\b \, \in  \, \B } \pn^{2}  (\b)\,\ind_{\Bn}(\b) = \int_{0}^{1}F\,\pn 
+ \sum_{\b \, \in \, \B} f(\b) \,  \pn(\b) \, \ind_{\Bn}(\b) .
\end{equation*}
Letting $n\rightarrow\infty$ in the expression above we get
\begin{equation}\label{Ineq Solution Sandwich}
\begin{split}
\lim_{n}\Big\{\int_{0}^{1}\del \pn^{2} 
+ \beta \sum_{\b \, \in  \, \Bn } \pn^{2}  (\b)\Big\}
& = \int_{0}^{1}F\,p 
+ \sum_{\b \, \in \, \B} f(\b) \,  p(\b)  \\
& = \int_{0}^{1}\del p^{2} 
+ \beta \sum_{\b \, \in  \, \B } p^{2}  (\b) \\
& \leq \liminf_{n}\int_{0}^{1} \left\vert \partial \pn\right\vert^{2}
+  \beta\, \liminf_{n}  \sum_{\b\,\in\, \Bn}\pn^{2}(\b) .
\end{split}
\end{equation}
Hence,
\begin{equation*}
\lim_{n}\Big\{\int_{0}^{1}\del \pn^{2} 
+ \beta \sum_{\b \, \in  \, \Bn } \pn^{2}  (\b) \Big\}
= \liminf_{n}\int_{0}^{1} \left\vert \partial \pn\right\vert^{2}
+  \beta\, \liminf_{n}  \sum_{\b\,\in\, \Bn}\pn^{2}(\b) .
\end{equation*}
From here, a simple exercise of real sequences shows that both sequences $\{\int_{0}^{1} \left\vert \partial \pn\right\vert^{2}: n\in \N \}$ and $\big\{\sum\limits_{\b\,\in\, \Bn}\pn^{2}(\b): n\in \N\big\}$ converge.
Combining these facts with \textsc{Inequality} \eqref{Ineq Solution Sandwich} we have
\begin{equation*} 
\int_{0}^{1}\del p^{2} 
+ \beta \sum_{\b \, \in  \, \B } p^{2}  (\b) \\
= \lim_{n}\int_{0}^{1}\del \pn^{2} 
+ \beta \, \lim_{n} \sum_{\b \, \in  \, \Bn } \pn^{2}  (\b).
\end{equation*}
Therefore, if $\displaystyle \int_{0}^{1}\del p^{2} \lneqq \liminf\limits_{n}\int_{0}^{1} \left\vert \partial \pn\right\vert^{2}$ or $\sum\limits_{\b \, \in  \, \B } p^{2}  (\b)  \lneqq \liminf\limits_{n} \sum\limits_{\b \, \in  \, \Bn } \pn^{2}  (\b)$, the equality above would not be possible. Then, it holds that
\begin{equation*}
\Vert p \Vert_{\scriptscriptstyle V}^{2}  
= \lim_{n}\int_{0}^{1}  \partial \pn^{2} 
= \lim_{n} \Vert \pn \Vert_{\scriptscriptstyle V}^{2} ,
\end{equation*}
and
\begin{equation*}
\Vert \gamma_{\scriptscriptstyle \B}(p) \Vert_{\scriptscriptstyle \ell^{2}(\B)}^{2}  
%
= \lim_{n} \sum_{\b\,\in\, \B}\pn^{2}(\b)  \, \ind_{\Bn}(\b) 
= 
\lim_{n} \Vert \gamma_{\scriptscriptstyle \B}(\pn\ind_{\Bn}) \Vert_{\scriptscriptstyle \ell^{2}(\B)}^{2}.
\end{equation*}
Finally, the convergence of norms together with the weak convergence on the underlying spaces, imply the strong convergence \eqref{Stmt strong convergence} of both sequences.
\qed
\end{proof}
%
%
%
%
%
\subsection{The Space $\Vb$}
We start this section proving a lemma which is central in the understanding of the space $\Vb$. 
\begin{lemma}\label{Th microstructure accumulation points}
Let $\B$ be a microstructure in $[0,1]$ and let $q\in H^{1}(0,1)$ be such that $q\vert_{\B}\in \ell^{t}(\B)$ with $1\leq t < \infty$. Then, if $\b_{0}\in \B$ is an accumulation point of $\B$, it must hold that $q(\b_{0}) = 0$.
\end{lemma}
\begin{proof}
Let $q$ satisfy the hypotheses and let $\b_{0}\in \B$ be an accumulation point of $\B$ such that $\vert q(\b_{0}) \vert >0$. Since $q\in H^{1}(0,1)$, it is absolutely continuous and there exists $\epsilon >0$ such that $\vert q(x)\vert > \dfrac{1}{2}\, \vert q(\b_{0}) \vert$ for all $\vert x - \b_{0}\vert < \epsilon$; hence
\begin{equation*}
\sum_{\b \, \in \, \B} \vert q(\b) \vert^{t} \geq 
\sum_{\substack{\b\, \in \,\B\\[3pt] \vert \b - \b_{0}\vert < \epsilon }} \vert q(\b) \vert^{t}
\geq 
 \card \big(\{ \b\, \in \,\B : \ \vert \b - \b_{0}\vert < \epsilon\}\big) \, \frac{1}{2^{p}}\,\vert q(\b_{0}) \vert^{t}.
\end{equation*}
The set $\{ \b\, \in \,\B : \ \vert \b - \b_{0}\vert < \epsilon\}$ contains infinitely many points because $\b_{0}$ is an accumulation point of $\B$, therefore $q\vert_{\B}$ does not belong to $\ell^{t}(\B)$ which is absurd.
\qed
\end{proof}
\begin{corollary}\label{Th dense microstructure}
Suppose that $\B$ is dense in $[0,1]$ then $\Vb = \{0\}$.
\end{corollary}
\begin{proof}
Due to \textsc{Lemma} \ref{Th microstructure accumulation points} if $q\in \Vb$ it must hold that $q(\b) = 0$ for every $\b$ accumulation point of $\B$. Since $\B$ is dense in $(0,1)$ every point of $\B$ is an accumulation point of $\B$, therefore $q(\b) = 0$ for all $\b\in \B$. On the other hand, $q$ is absolutely continuous because it belongs to $H^{1}(0,1)$ and due to the density of $\B$ in $[0,1]$ it follows that $q = 0$.
\qed
\end{proof}
The lemma \ref{Th microstructure accumulation points} states that the accumulation points of the microstructure contained in $\B$ define an important property of $\Vb$. In addition, the next property follows trivially
\begin{corollary}\label{Th dense microstructure trivializes}
Let $\B\subseteq [0,1]$ be a dense microstructure then, \textsc{Problem} \eqref{Pblm Variational Microstructure Concise} becomes trivial and the sequence $\{\pn: n\in \N \}$ satisfies
\begin{align}\label{Stmt convergence to zero dense microstructure}
& \Vert \pn\Vert_{\scriptscriptstyle V} \rightarrow 0 \; , & 
& \sum_{\b \, \in \, \Bn}\pn^{2}(\b) \rightarrow 0 \; .
\end{align}
\end{corollary} 
%
%
The facts presented in \textsc{Lemma} \ref{Th microstructure accumulation points} and \textsc{Corollary} \ref{Th dense microstructure trivializes} are unfortunate, since several important fractal microstructures are dense in $[0,1]$. On the other hand, if a microstructure $\B$ is not dense in $[0,1]$ but it is a perfect set (which is also an important case) the problem \eqref{Pblm Variational Microstructure Concise}, without becoming trivial, becomes fully decoupled and consequently uninteresting. An important example of the first case are the Dyadic numbers in $[0,1]$ and an example of the second case is the collection of extremes from the removed intervals in the construction of the Cantor set. 

As an alternative, it is possible to strengthen the conditions on the interface forcing term $f$, seeking to weaken the summability properties of the limit function fractal trace $\gamma_{\scriptscriptstyle \B}(p) = p\ind_{\B}$. However, this approach yields estimates equivalent to \textsc{Inequality} \eqref{Ineq a-priori estimate pressures} in \textsc{Theorem} \ref{Th boundedness of the sequence of solutions} and consequently $p\in \Vb$. Therefore, this is not a suitable choice either.

\section{The Fractal Scaling Model}\label{Dec Fractal Scaling Modeling}
%
%
%
%
The unsatisfactory conclusions shown in the previous section are, essentially, due to a physical fact assumed in the model: that the storage fluid exchange coefficient $\beta$ is constant all over the microstructure $\B$. Consequently, the storage effect across the microstructure adds up to infinity. Hence, the modeling of $\beta$ has to avoid this hypothesis. On one hand, we need to assure that the form $\Lambda: V\times V\rightarrow \R$, defined by
\begin{equation}\label{Def Bilinear Form Scaled Coefficient}
\Lambda(q, r)\defining 
\sum_{\b \, \in  \, \B} \beta(\b) \, q  (\b) \,  r(\b),
\end{equation}
%
%
is bilinear and continuous. On the other hand, \textsc{Lemma} \ref{Th microstructure accumulation points} states the need to avoid global estimates for $\{\pn(\b)\ind_{\Bn}(\b): \b\in \B \}\subseteq \ell^{t}(\B)$ for any $t > 1$. Therefore, if $\{\beta(\b): \b\in \B\} \in \ell^{1}(\B)$ and recalling that $\beta(x) \geq 0$ for all $x\in (0,1)$, the bilinear form satisfies
\begin{equation} \label{Def Bilinear Form Scaled Coefficient Continuity}
\begin{split}
\bigg\vert  \sum_{\b \, \in  \, \B} \beta(\b) \,  q  (\b)  \,  r(\b) \bigg\vert
&  \leq \sup_{\b\,\in\,\B} \vert  q  (\b) \vert  \,  \sup_{\b\,\in\,\B} \vert r(\b) \vert  \sum_{\b \, \in  \, \B}  \beta(\b)  \\
& \leq \Vert q \Vert_{V} \, \Vert r \Vert_{V} \sum_{\b \, \in  \, \B}  \beta(\b) .
 \end{split}
\end{equation}
In order to attain this condition, it is natural to assume that the storage coefficient $\beta$, scales consistently with the properties of the fractal microstructure. This motivates the following definition
\begin{definition}\label{Def scaled storage coefficient}
Let $\B\subseteq [0,1]$ be a fractal microstructure with $L > 1$ as given in \textsc{Definition} \ref{Def Fractal Microstructure}. Then, a storage coefficient $\beta: \B\rightarrow (0, \infty)$, is said to \textbf{scale consistently} with a given $\sigma$-finite development $\{\Bn: n\geq 0 \}$ if it satisfies
\begin{equation}\label{Eq scaled storage coefficient}
\beta (\b) \defining  a \sum_{n \, \in \, \N} \Big(\frac{1}{L} - \epsilon \Big)^{n}   
\, \ind_{\Bn - \B_{n-1}}(\b)\, ,
\quad \text{with}\; a > 0 \; \text{and}\;\,  0 < \epsilon<  \frac{1}{L} \, .
\end{equation}
\end{definition}
\begin{proposition}\label{Th scaled storage coefficient convergence}
Let $\beta: \B\rightarrow (0, \infty)$ be a storage coefficient consistently scaled with $\B$ then $\beta\in \ell^{1}(\B)$.
\end{proposition}
\begin{proof}
Since $\{\Bn:n \geq 0\}$ is the $\sigma$-finite development of $\B$ with $\B_{0}\neq\emptyset$, the cardinality \textsc{Identity} \eqref{Eq Cardinality Relationships} implies
%
\begin{align*} 
\sum_{\b\,\in\,\Bn}\vert \beta(\b) \vert = 
\sum_{\b\,\in\,\Bn} \beta(\b) 
& = 
\sum_{k \, = \, 1}^{n} \; \sum_{\b \, \in \, \Bk - \Bkl} \beta(\b) \\
& =
\sum_{k \, = \, 1}^{n} a \, L^{k-1}(L-1) \, \card(\B_{0}) \Big(\frac{1}{L} - \epsilon \Big)^{k}
= a   \, \card(\B_{0})\big(1 - \frac{1}{L}\big) \sum_{k \, = \, 1}^{n} \big(1 - \epsilon \, L\big)^{k}  .
\end{align*}
Since $\epsilon\in \big(0, \dfrac{1}{L} \big)$, as stated in \textsc{Definition} \ref{Def scaled storage coefficient}, the expression above is convergent and the result follows. 
\qed
\end{proof}
\begin{remark}
It is immediate to see some variations of \textit{Definition} \ref{Def scaled storage coefficient} based on the geometric series properties. For instance, if $\{\alpha_{n}: n\in \N\}\subseteq (0,\infty)$ is a sequence such that $\displaystyle \limsup\limits_{ n } \sqrt[n]{\alpha_{n}} < \frac{1}{L}$, then take
\begin{equation*} 
\beta (\b) \defining   \sum_{n \, \in \, \N} 
\alpha_{n}\, \ind_{\Bn - \B_{n-1}}(\b) .
\end{equation*}
The storage coefficient defined above will also satisfy that $\beta\in \ell^{1}(\B)$ and permit the desired \textsc{Estimate} \eqref{Def Bilinear Form Scaled Coefficient Continuity}. A more sophisticated variation considers probabilistic uncertainty for the values $\{\alpha_{n}: n \in \N\}$, whether on its decay rate or, on a distribution centered at the self-similarity parameter $L$. A very basic probabilistic version of the latter will be numerically illustrated in \textsc{Section} \ref{Sec Numerical Experiments}.
\end{remark}
%
%
%
\subsection{The Limit Problem}
%
%
In the following it will be assumed that $F\in L^{2}(0,1)$. 
For notational simplicity it is understood that both, the storage coefficient and the interface forcing term are defined on the whole domain $[0,1]$, with $\beta\big\vert_{(0,1) - \B} = f\big\vert_{(0,1) - \B} = 0$.  It is also assumed that the storage coefficient $\beta\big\vert_{\B}$ scales consistently with the fractal microstructure $\B$ and that the forcing term is summable i.e., $f\in \ell^{1}(\B)$.
\begin{theorem}\label{Th well posedness good coefficient problems}
The following problems are well-posed.  
\begin{align}\label{Pblm Variational n stage Good Coefficients}
& \pn \in V : &
& \int_{0}^{1}\del \pn \, \del q 
+  \sum_{\b \, \in  \, \Bn} \beta(\b)\, \pn  (\b) \,  q(\b)= \int_{0}^{1}F\,q 
+ \sum_{\b \, \in \, \Bn} f(\b) \,  q(\b)\,, &
& \forall q\in V.
\end{align}
\begin{align}\label{Pblm Variational Microstructure Good Coefficients}
& p \in V : & 
& \int_{0}^{1}\del p \, \del q 
+  \sum_{\b \, \in  \, \B} \beta(\b)\, p (\b) \,  q(\b)= \int_{0}^{1}F\,q 
+ \sum_{\b \, \in \, \B} f(\b) \,  q(\b)\,, &
& \forall q\in V.
\end{align}
\end{theorem}
\begin{proof}
Consider the bilinear forms
\begin{subequations} \label{Def Bilinear Forms Scaled Coefficient}
\begin{equation}\label{Def Bilinear Forms Scaled Coefficient n stage}
\langle q, r\rangle_{\scriptscriptstyle \Bn}  \defining 
\int_{0}^{1}\del q \, \del r 
+  \sum_{\b \, \in  \, \Bn} \beta(\b)\, q  (\b) \,  r(\b) \, ,  \quad n\in \N, 
\end{equation}
\begin{equation}\label{Def Bilinear Forms Scaled Coefficient full Microstructure}
\langle q, r\rangle_{\scriptscriptstyle \B}  \defining \int_{0}^{1}\del q \, \del r 
+  \sum_{\b \, \in  \, \B} \beta(\b)\, q  (\b) \,  r(\b) .
\end{equation}
\end{subequations}
Since $\beta\in \ell^{1}(\B)$, \textsc{Inequality} \eqref{Def Bilinear Form Scaled Coefficient Continuity} is satisfied and the bilinear forms $(q, r)\mapsto \sum\limits_{\b \, \in  \, \B} \beta(\b)\, q  (\b) \,  r(\b)$, $(q, r)\mapsto \sum\limits_{\b \, \in  \, \Bn} \beta(\b)\, q  (\b) \,  r(\b)$ for each $n\in \N$  
are continuous; consequently both bilinear forms \eqref{Def Bilinear Forms Scaled Coefficient} are continuous.  On the other hand, since the coefficient $\beta$ is non-negative the bilinear forms \eqref{Def Bilinear Forms Scaled Coefficient} are both coercive on $V$. Applying the Lax-Milgram Theorem the result follows.
\qed
\end{proof}
\begin{theorem}
Let $\{\pn: n\in \N \}$ be the sequence of solutions of \textsc{Problems} \eqref{Pblm Variational n stage Good Coefficients} and let $p$ be the solution to \textsc{Problem} \eqref{Pblm Variational Microstructure Good Coefficients}. Then 
\begin{enumerate}[(i)]
\item $\{\pn: n\in \N \}$  converges weakly to $p$ in $V$.

\item The following convergence statements hold
\begin{subequations}\label{Stmt Convergence Interface Functionals}
\begin{equation} \label{Stmt Strong Convergence Non-Linear Interface Functionals}
\sum_{\b \, \in  \, \Bn} \beta(\b)\, \pn^{2} (\b) \xrightarrow[n\rightarrow \infty]{}
\sum_{\b \, \in  \, \B} \beta(\b)\, p^{2} (\b).
\end{equation}
\begin{equation} \label{Stmt Strong Convergence Forcing Interface Functionals}
\sum_{\b \, \in  \, \Bn} f(\b)\, \pn (\b) \xrightarrow[n\rightarrow \infty]{}
\sum_{\b \, \in  \, \B} f(\b)\, p (\b).
\end{equation}
\end{subequations}

\item $\{\pn: n\in \N \}$  converges strongly to $p$ in $V$.
\end{enumerate}
\end{theorem}
\begin{proof}
\begin{enumerate}[(i)]
\item The result follows using the techniques presented in \textsc{Theorem} \ref{Th boundedness of the sequence of solutions} yield the boundedness of the sequence $\{\pn: n\in \N \}$. Next, using the reasoning of \textsc{Theorem} \ref{Th weak convergence of the solutions in V} the weak convergence follows. 

\item Due to the weak convergence of the solutions and the fact that the evaluation is a continuous linear functional, it follows that $\pn(\b) \rightarrow p(\b)$ and $\pn^{2}(\b) \rightarrow p^{2}(\b)$ for all $\b\in \B$. Let $M> 0$ be such that $\Vert p \Vert_{V}^{2}\leq M$ and $\Vert \pn \Vert_{V}^{2}\leq M$ for all $n\in \N$. Fix $k\in \N$ such that $n > k$ implies $\Big\vert \sum\limits_{\B - \Bn } \beta(\b)\Big\vert \leq \dfrac{\epsilon}{3\, M^{2}}$, which we know to exist since $\beta\in \ell^{1}(\B)$. Then, for any $n > k$ it holds that
\begin{equation*} 
\begin{split}
\bigg\vert \sum_{\b \, \in  \, \Bn} \beta(\b)\, \pn^{2} (\b) -
\sum_{\b \, \in  \, \B}  & \beta(\b)\, p^{2} (\b) \bigg\vert \\
& \leq  
\sum_{\b \, \in  \, \Bk} \beta(\b)\,\left\vert \pn^{2} (\b) -
 p^{2} (\b) \right\vert + 
 \sum_{\b \, \in  \, \Bn - \Bk} \beta(\b)\, \left\vert\pn^{2} (\b) \right\vert 
+  \sum_{\b \, \in  \, \B - \Bk} \beta(\b)\, \left\vert
p^{2} (\b) \right\vert\\
& \leq 
\sum_{\b \, \in  \, \Bk} \beta(\b)\,\left\vert \pn^{2} (\b) -p^{2} (\b) \right\vert 
+ \Vert
 \pn\Vert_{V}^{2}    \sum_{\b \, \in  \, \B - \Bk} \beta(\b)\, 
 + \Vert
   p\Vert_{V}^{2} \sum_{\b \, \in  \, \B - \Bk} \beta(\b) \\
& \leq  
\sum_{\b \, \in  \, \Bk} \beta(\b)\,\left\vert \pn^{2} (\b) -p^{2} (\b) \right\vert 
+ \frac{2}{3}\, \epsilon . \\
%
%
%
\end{split}
\end{equation*}
Since $k\in \N$ is fixed, choose $N\in \N$ such that $n\geq N$ implies 
\begin{align*}
& \vert \pn^{2}(\b) - p^{2}(\b) \vert \leq \frac{\epsilon}{3} \;
\big\Vert \beta \big\Vert_{\ell^{1}(\B)}^{-1} ,&
& \text{for all } \b\in \Bk.
\end{align*}
Hence, combining with the previous estimate, the convergence statement \eqref{Stmt Strong Convergence Non-Linear Interface Functionals} follows. For the statement \eqref{Stmt Strong Convergence Forcing Interface Functionals} it suffices to combine the strong convergence of the forcing terms $\big\Vert f\ind_{\Bn} -f\big\Vert_{\ell^{1}(\B)}\xrightarrow[n \, \rightarrow \, \infty]{}0$ with the weak convergence of the solutions $\pn \wconv p$. This completes the second part.

\item Again, the result follows combining the convergence statements \eqref{Stmt Strong Convergence Non-Linear Interface Functionals} and \eqref{Stmt Strong Convergence Forcing Interface Functionals}, with the techniques presented in \textsc{Lemma} \ref{Th norms convergence}, used to attain the strong convergence of the solutions. 
%
%
\qed
\end{enumerate}
\end{proof} 
Next, we present the closest version of a strong form of \textsc{Problem} \eqref{Pblm Variational Microstructure Good Coefficients}.
%
%
\begin{theorem}\label{Th pseudo strong form}
Let $\B\subseteq [0,1]$ be a microstructure and let $\B'$ be its set of limit points.
Then, \textsc{Problem} \eqref{Pblm Variational Microstructure Good Coefficients} is a weak formulation of the following strong problem.
\begin{subequations}\label{Pblm Strong Microstructure Good Coefficients}
\begin{equation} \label{Eq limit laplacian Microstructure Good Coefficients}
%
-\del^{2} p 
= F \quad  \text{in the sense } \; \;L^{2} \big((0,1) - \B ' 
\big) , 
%
\end{equation}
\begin{equation} \label{Eq boundary conditions limit problem}
p (0) = 0 . 
\end{equation}
Together with the interface fluid transmission conditions for isolated points of $\B$
%
    \begin{equation}\label{Eq strong limit interface isolated points}      
    \begin{split}
    p \left(\bnot^{-}\right) & = p\left(\bnot^{+}\right)  , 
    \\
    \del p \left( \bnot^{-}\right)  -  \del p \left( \bnot^{+}\right) + 
   \beta (\bnot) \,  p \left( \bnot\right) & =  f\left( \bnot\right) \, ,
    \quad \forall\, \bnot \in \B - \B'.
    \end{split}
    \end{equation}
And the non-localizable fluid transmission conditions for limit points of $\B$
%
%
%
%
\begin{equation}\label{Eq strong limit interface accumulation points extra hypothesis} 
\begin{split} 
 p \left(x^{-}\right) & = p\left(x^{+}\right)  , 
    \\    
\lim\limits_{n \, \rightarrow \, \infty}
\frac{1}{\frac{1}{n}}
\int_{x - \frac{1}{n}}^{x + \frac{1}{n}}\del p 
+  \beta(x)\, p (x) \ind_{\B}(x) & = 
f(x) \, .
\end{split} 
\end{equation}
 
\end{subequations}
\end{theorem}
\begin{proof}
The boundary condition  \eqref{Eq boundary conditions limit problem} holds because $p \in V$. Let $x\in (0,1) - \B '$, then, there exists $\delta_{0}> 0 $ such that $(x - \delta, x+ \delta)\cap \B = \{x\}\cap \B$ for all $\delta \in (0, \delta_{0})$. Test \textsc{Problem} \eqref{Pblm Variational Microstructure Good Coefficients} with $\varphi\in C_{0}^{\infty}(x-\delta, x)$ to get
\begin{equation*}
-\big\langle \del^{2} p, \, \varphi\big\rangle_{\D '(x - \delta, x), D(x - \delta, x)}
= \int_{x - \delta}^{x }   \del p \, \del \varphi  =
\int_{x - \delta}^{x } F\varphi .
\end{equation*}
Since the above holds for every smooth function with support contained in $(x - \delta, x)$, we conclude that $- \del^{2} p = F$ in $L^{2}(x - \delta, x)$. Similarly, it follows that $- \del^{2} p = F$ in $L^{2}(x, x + \delta)$, therefore $\del p\in H^{1}(x - \delta, x)\cap H^{1}(x, x+ \delta)$ and the statement \eqref{Eq limit laplacian Microstructure Good Coefficients} follows. Now, choose $\varphi \in C_{0}^{\infty}(x - \delta, x + \delta)$ and test \textsc{Problem} \eqref{Pblm Variational Microstructure Good Coefficients} to get
%
\begin{equation*}
\int_{x - \delta}^{x + \delta}   \del p \, \del \varphi  
+\beta(x)\, p(x) \, \varphi(x) 
= \int_{x - \delta}^{x + \delta} F\varphi 
+ f(x) \, p(x)  \, \varphi(x) .
\end{equation*}
In the expression above, break the interval conveniently in order to use integration by parts 
\begin{equation*}
\begin{split}
\int_{x - \delta}^{x + \delta} F\varphi 
+ f(x) \, p(x)  \, \varphi(x)
%
& = \int_{x - \delta}^{x }   \del p \, \del \varphi  
+ \int_{x }^{x + \delta}   \del p \, \del \varphi 
+\beta(x)\, p(x)  \, \varphi(x)  \\
& = - \int_{x - \delta}^{x }\del^{2} p \, \varphi
- \int_{x }^{x + \delta}   \del^{2} p \,  \varphi 
+ \del p\, \varphi \big\vert_{x - \delta}^{x}
+ \del p\, \varphi \big\vert_{x }^{x + \delta}
+\beta(x)\, p(x)  \, \varphi(x)  .
\end{split}
\end{equation*}
From the initial part of the proof, the first summand of the left hand side, cancels with the first two summands of the right hand side. Evaluating the boundary terms and recalling that $\varphi (x - \delta) = \varphi(x + \delta) = 0$ this gives
\begin{equation*} 
f(x) = \del p(x^{-})
- \del p(x^{+})
+   \beta(x)\, p (x).
\end{equation*}
Consequently, the normal flux balance condition in \eqref{Eq strong limit interface isolated points} follows for points $x\in \B -\B'$. The normal stress balance condition in \eqref{Eq strong limit interface isolated points} follows from the continuity of $p$ across the interface, which holds for any function in $V\subseteq H^{1}(0,1)$. The normal stress balance in \textsc{Statement} \eqref{Eq strong limit interface accumulation points extra hypothesis} holds due to the continuity of $p$ at any point of $x\in (0,1)$. Now, fix $x\in \B'$ and test the problem \eqref{Pblm Variational Microstructure Good Coefficients} with $q_{n}(t)\defining q\big(n(t - x)\big)$, where $q(t) \defining (t +1)\ind_{(-1,0)}(t) + (1 - t)\ind(0,1)(t)$; this gives
\begin{equation*} 
\frac{1}{\frac{1}{n}}\int_{x - \frac{1}{n}}^{x+\frac{1}{n}}\del p 
+  \sum_{\b \, \in  \, \B} \beta(\b)\, p (\b) \,  q_{n}(\b)
= \int_{x - \frac{1}{n}}^{x+\frac{1}{n}} F\,q_{n} 
+ \sum_{\b \, \in \, \B} f(\b) \,  q_{n}(\b) .
\end{equation*}
Taking limits in the expression above and recalling the Lebesgue Dominated Convergence Theorem, the result follows.    
\qed
\end{proof}
%
%
%
%
Notice that for the solution of \textsc{Problem} \eqref{Pblm Variational Microstructure Good Coefficients}, the following equality holds in the sense of distribution by definition 

%
\begin{align*}
& -\del^{2} p + \sum_{\b \, \in  \, \B} \beta(\b)\, p (\b) \,  \delta_{\{\b\}}  
= F + \sum_{\b \, \in \, \B} f(\b) \delta_{\{\b\}}
 \, , &
& \text{in }
\D '  (0, 1) .
\end{align*}
Where $\delta_{\{\b\}}(q) \defining q(\b)$ is the Dirac evaluation functional. 
It is easy to observe that if $x$ is a limit point of $\B$, both series in the expression above will pick up infinitely many non-null terms for test functions satisfying $\supp(\varphi)\cap \B' \neq \emptyset$. Therefore, the techniques used in \textsc{Theorem} \ref{Th pseudo strong form} to derive point-wise statements from weak variational ones, do not apply. Again, this fact is unfortunate because, as pointed out at the end of \textsc{Section} \ref{Sec A sigma finite Interface Microstructure}, several important fractal microstructures are Notice that, by definition of distributions, the following equality holds in the sense of distribution for the solution of \textit{Problem} \eqref{Pblm Variational Microstructure Good Coefficients} or dense in $[0,1]$. If $\B$ is dense, very little can be said about the point-wise behavior at any point of the domain and if $\B$ is a perfect set, the transmission conditions can not be described locally at any point of the interface. A positive aspect of this model is that the aforementioned cases do not become trivial or uninteresting. Additionally, if the set of accumulation points $\B '$ has null Lebesgue measure the non-localizable interface conditions are not relevant for the global phenomenon. Fortunately, this is the case for very important fractal microstructures e.g., the collection of extremes of the removed intervals in the construction of the Cantor set, the Sierpinski Triangle in 2-D, etc.  
%
%
%
%
%
%
\section{Numerical Experiments}\label{Sec Numerical Experiments}
%
%
%
%
In this section we present two types of numerical experiments. The first type are verification examples, supporting our homogenization conclusions for a problem whose asymptotic behavior is known exactly.  The second type are of exploratory nature, in order to gain heuristic understanding of the probabilistic variations of the model. The experiments are executed in a MATLAB code using the Finite Element Method (FEM) which is an adaptation of the code \textbf{fem1d.m} \cite{PeszynskaFEM}.
%
%
\subsection{General Setting}
%
%
For the sake of simplicity, in all the examples below, the microstructure $\B\subseteq [0,1]$ is the collection of extremes from the removed intervals in the construction of the Cantor set. In addition, its $\sigma$-finite development $\{\Bn: n\geq 0 \}$ is the natural one i.e., 

\begin{align}\label{Def Partial Triadic Extremes Set}
\begin{split}
\B_{0} \defining \big\{0, 1 \big\}&  , \\
%
\B_{n+1}\defining \Big\{\frac{1}{3}\, \b: \b\in \Bn  \Big\} \cup 
\Big\{1 - \frac{1}{3}\, \b: \b\in \Bn  \Big\} &,\quad \forall \, n\geq 1 . 
\end{split}
\end{align}
For the experiments, it will be shown that the sequences $\{\pn: n\in \N \}$ are Cauchy ($n\in \N$ indicates the stage of the Cantor set construction). In all the cases, the computations are made for the stages $6, 7,8, 9$. For each example, graphics of the solution for the $n$-stages chosen from $\{3,4,6,9\}$ are displayed, based on optical neatness: two or three graphs for the solution together with their corresponding derivatives. In addition, for visual purposes vertical lines in the derivatives' graphics are introduced, to highlight the fractal structure of the function. For all the examples, the forcing term $F\in L^{2}(0,1)$ is set equal to zero and the interface forcing term $f: \B\rightarrow \R$, is given by the following expression 

\begin{equation}\label{Def Deterministic Interface Forcing Term}
f (\b) \defining \sum_{n \, \in \, \N} \frac{1}{3^{n}}   
\, \ind_{\Bn - \B_{n-1}}(\b)\, .
\end{equation}
%
%
%
%
%
\subsection{The Examples}
%
%
%
%
\begin{example}[Unscaled Example]\label{Ex Basic Example}
In the first example we implement the model studied in \textsc{Section} \ref{Sec A sigma finite Interface Microstructure}, in order to verify its conclusions. In particular, it must hold that $\Vert \pn\Vert_{V} \rightarrow 0$. We set the storage coefficient $\beta = 1$ and the forcing terms $F = 0$, $f$ defined in \eqref{Def Deterministic Interface Forcing Term} above. The convergence is presented \textsc{Table} \ref{Tbl Convergence Table Ex 1} below, together with the corresponding graphics in \textsc{Figure} \ref{Fig Solutions Ex1}. We observe convergence as established by the theoretical discussion of \textsc{Section} \ref{Sec A sigma finite Interface Microstructure}.

%
\begin{table*}[h!]
\def\arraystretch{1.4}
\begin{center}
\rowcolors{2}{gray!25}{white}
\begin{tabular}{c c c c c c}
    \hline
    \rowcolor{gray!50}
Stage $n$ &  $\#$ Nodes & $\Vert  p_{n}   \Vert_{L^{2}(0,1)} $ & $\Vert  p_{n}   \Vert_{H^{1}(0,1)}$
& $\Vert  p_{n}- p_{n - 1}  \Vert_{L^{2}(0,1)} $ & $\Vert  p_{n}- p_{n-1}  \Vert_{H^{1}(0,1)}$  \\ 
    \hline
6 & $3^{6} + 1$ &  0.011432576436 &  0.050569870333 & 0.008280418446366  & 0.022150608375185 \\
7  & $3^{7} + 1$ & 0.006705537618 &  0.040243882138 & 0.004790279224858 & 0.014569849160095 \\ 
8  & $3^{8} + 1$ & 0.004020895404 & 0.033172226789 & 0.002725704395092 & 0.010008013980805 \\ 
9  & $3^{9} + 1$ & 0.000334472470 &  0.000634665497 & 0.001567106680915 & 0.007680602752328 \\ 
    \hline
\end{tabular}
\caption{\textsc{Example} \ref{Ex Basic Example} : Convergence Table, $\beta = 1$.}\label{Tbl Convergence Table Ex 1}
\end{center}
\end{table*}
\begin{figure}[h!]
\vspace{-1.3cm}
        \centering
        \begin{subfigure}[Solution $p^{4}$, Stage $n = 4$.]
                { 
                {\includegraphics[scale = 0.45]{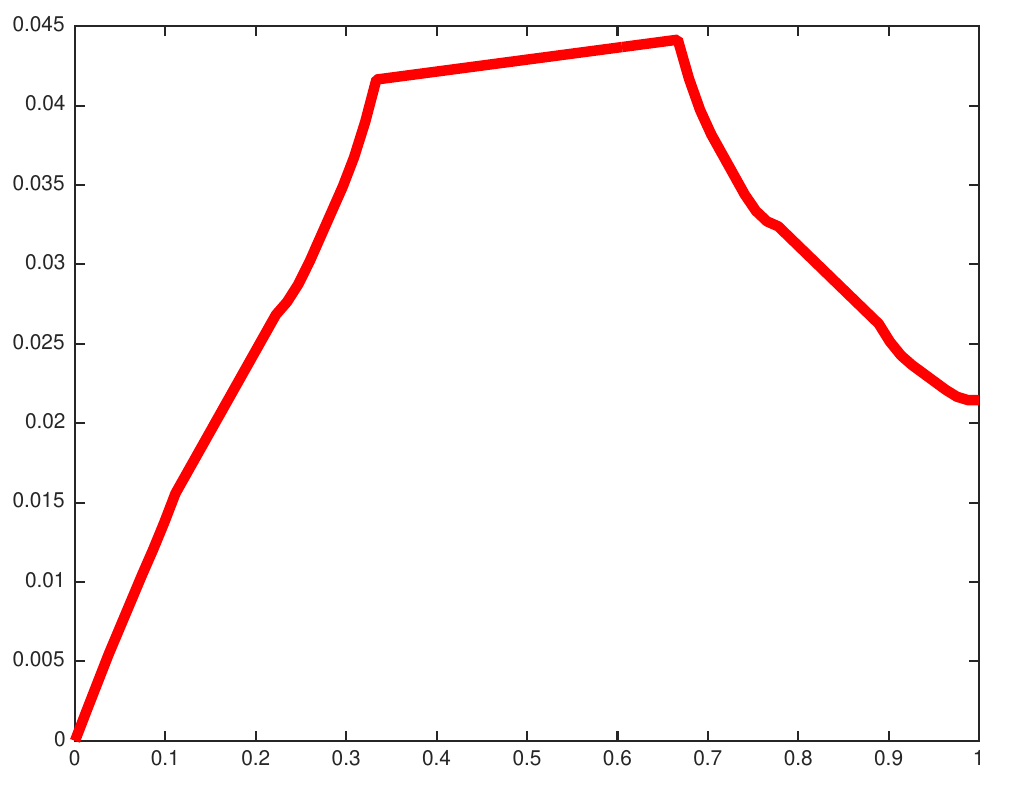} } }
        \end{subfigure} 
%
          \begin{subfigure}[Derivative $\del p^{4}$, Stage $n = 4$.]
                { 
                {\includegraphics[scale = 0.44]{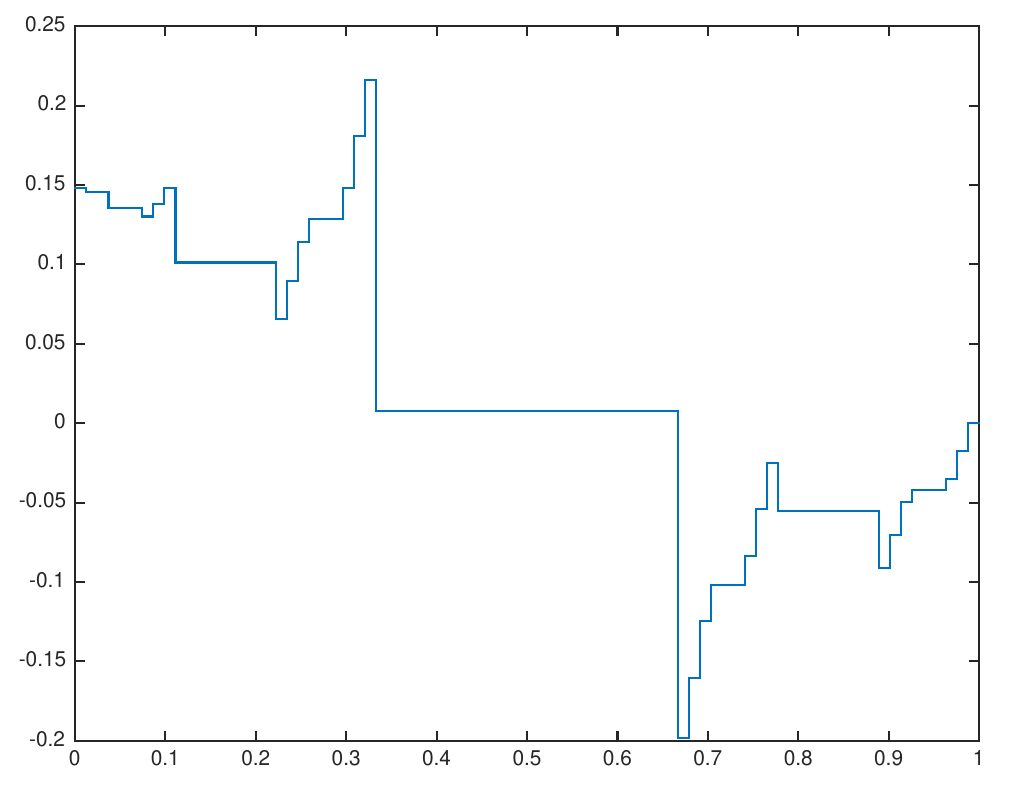} } }                
        \end{subfigure} 
 \\ 
        \begin{subfigure}[Solution $p^{9}$, Stage $n = 9$.]
                { 
{\includegraphics[scale = 0.45]{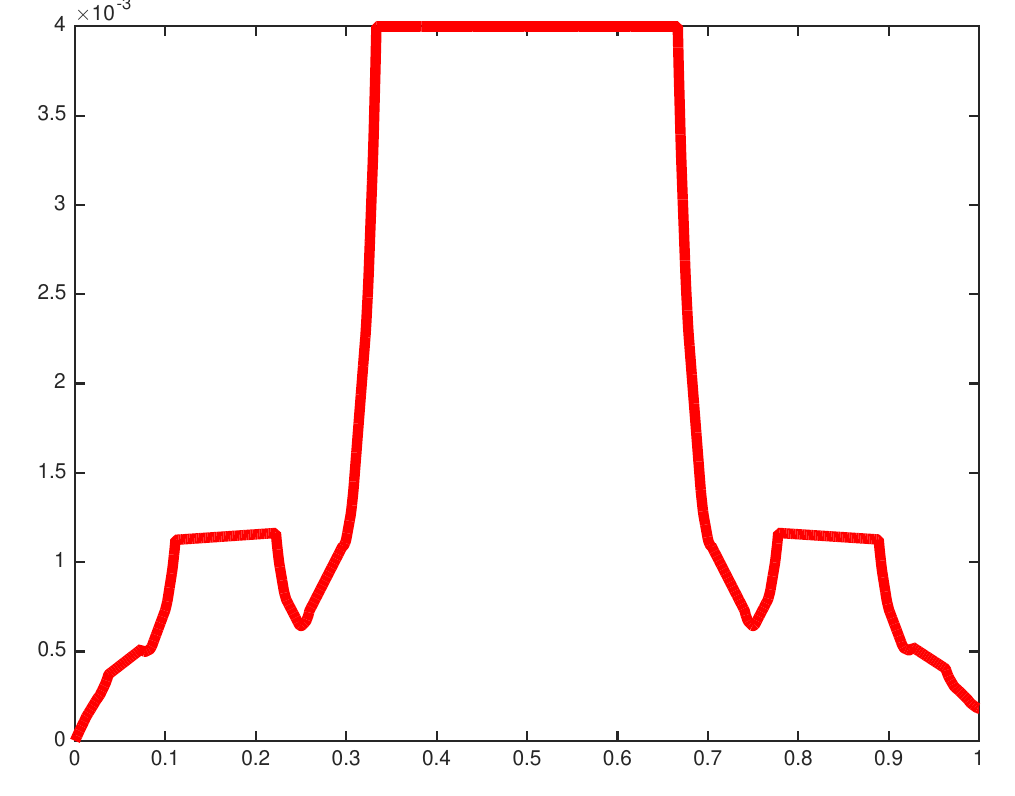} } }
        \end{subfigure}
%
        ~ 
          \begin{subfigure}[Derivative $\del p^{9}$, Stage $n = 9$.]
                { 
{\includegraphics[scale = 0.44]{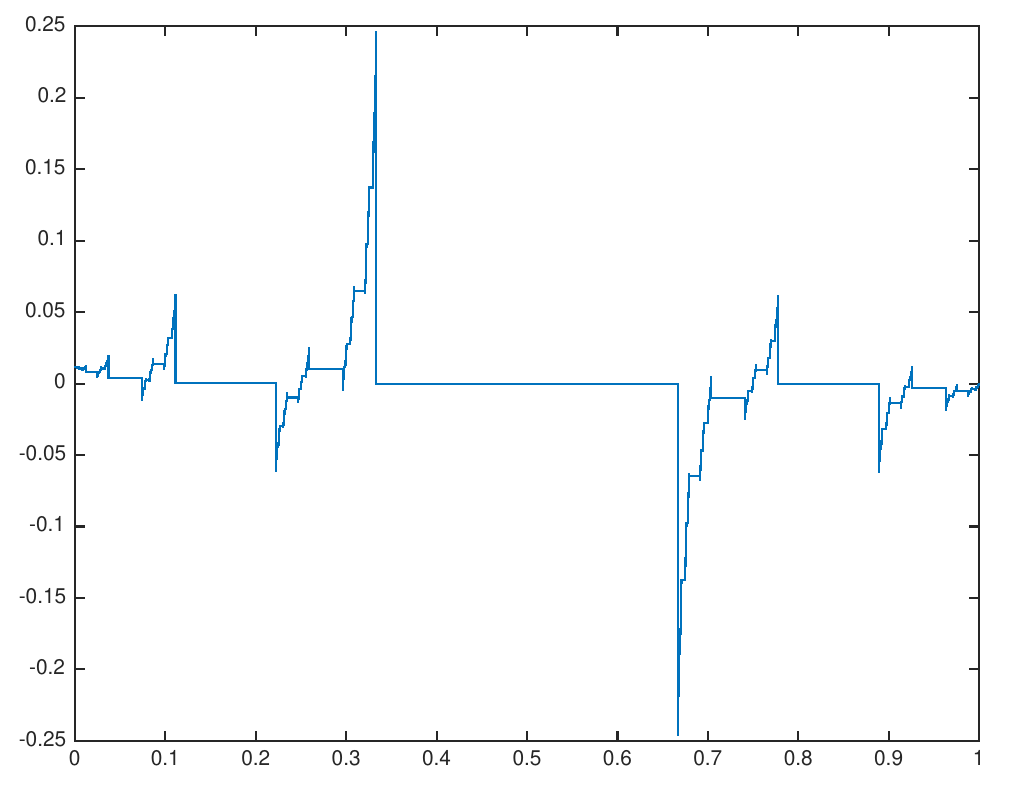} } }                
        \end{subfigure} 
\caption{\textsc{Solutions Example} \ref{Ex Basic Example}. Storage coefficient $\beta = 1$, forcing term $F = 0$ and interface forcing term $f$, defined in \textsc{Equation} \eqref{Def Deterministic Interface Forcing Term}. The functions depicted in figures (a) and (c) are respectively, the solutions when the microstructures $\B_{4}$ and $\B_{9}$ are present; see \eqref{Def Partial Triadic Extremes Set}. Figures (b) and (d) are the corresponding derivatives. The vertical lines in the derivatives' graphics are included for optical purposes only.}\label{Fig Solutions Ex1}
\end{figure}
\end{example}
\begin{example}[Scaled Example]\label{Ex Scaled Example} In the second example we implement the model introduced in \textsc{Section} \ref{Dec Fractal Scaling Modeling}, in order to illustrate the behavior of the solutions. We set $F = 0$, $f$ defined in \textsc{Equation}\eqref{Def Deterministic Interface Forcing Term} above and the storage coefficient

\begin{equation}\label{Eq Example scaled storage coefficient}
\beta (\b) \defining   \sum_{n \, \in \, \N} \big(\frac{2}{3} \big)^{n}   
\, \ind_{\Bn - \B_{n-1}}(\b)\, .
\end{equation}
Due to its fractal nature, the exact solution can not be described explicitly. Therefore, we only present the Cauchy behavior of the sequence of solutions in \textsc{Table} \ref{Tbl Convergence Table Ex 2} below, together with the corresponding graphics of the solutions in \textsc{Figure} \ref{Fig Solutions Ex2}. Again, convergence of the solutions is observed, as concluded by the theoretical results of \textsc{Section} \ref{Dec Fractal Scaling Modeling}.
%
%
\begin{table*}[h!]
\def\arraystretch{1.4}
\begin{center}
\rowcolors{2}{gray!25}{white}
\begin{tabular}{c c c c}
    \hline
     \rowcolor{gray!50}
Stage $n$  & $\#$ Nodes  & 
$\Vert  p_{n}- p_{n - 1}  \Vert_{L^{2}(0,1)} $ & $\Vert  p_{n}- p_{n-1}  \Vert_{H^{1}(0,1)}$  \\ 
    \hline
6  & $3^{6} + 1$ & 0.002708076748262 &  0.005143218742685 \\
7  & $3^{7} + 1$ & 0.001344770703905 &  0.002552130987523 \\ 
8  & $3^{8} + 1$ & 0.000670086347099 &  0.001271522606439\\ 
9  & $3^{9} + 1$ & 0.000334472470474 &  0.000634665496467 \\ 
    \hline
\end{tabular}
\caption{\textsc{Example} \ref{Ex Scaled Example}: Convergence Table, $\beta$ defined in \textsc{Equation} \eqref{Eq Example scaled storage coefficient}.}\label{Tbl Convergence Table Ex 2}
\end{center}
\end{table*}
\begin{figure}[h!] 
\vspace{-1.3cm}
        \centering
        \begin{subfigure}[Solution $p^{3}$, Stage $n = 3$.]
                { 
                {\includegraphics[scale = 0.45]{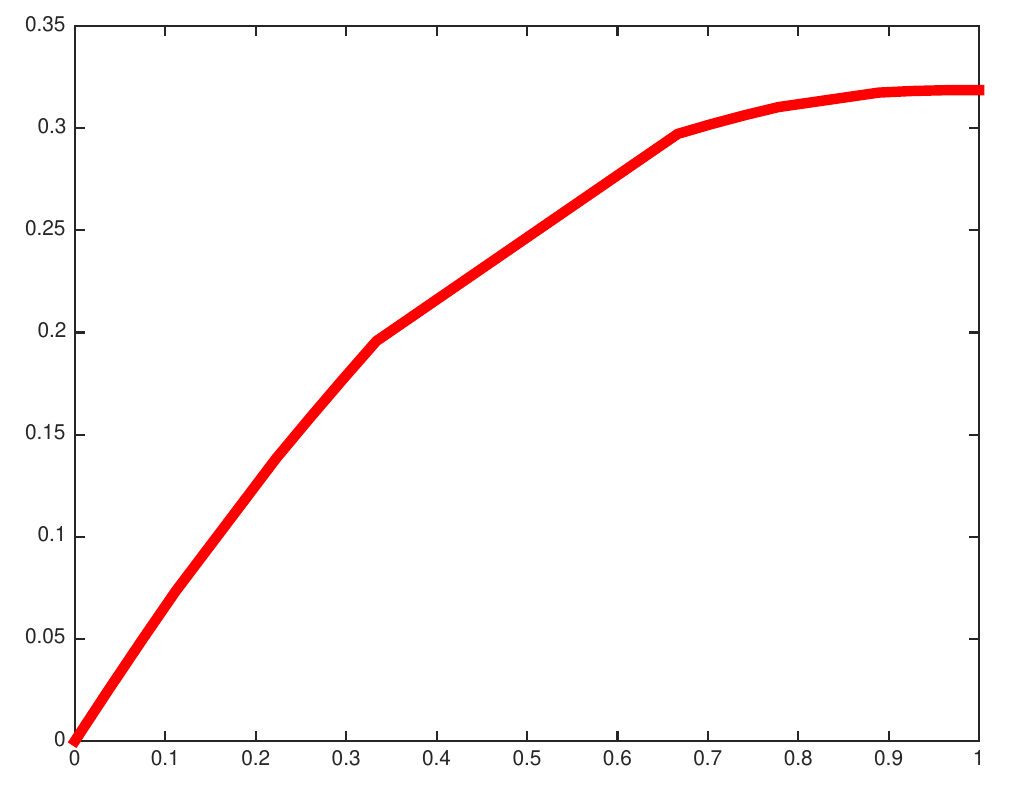} } }
        \end{subfigure} 
%
        \begin{subfigure}[Derivative $\del p^{3}$, Stage $n = 3$.]
                { 
{\includegraphics[scale = 0.44]{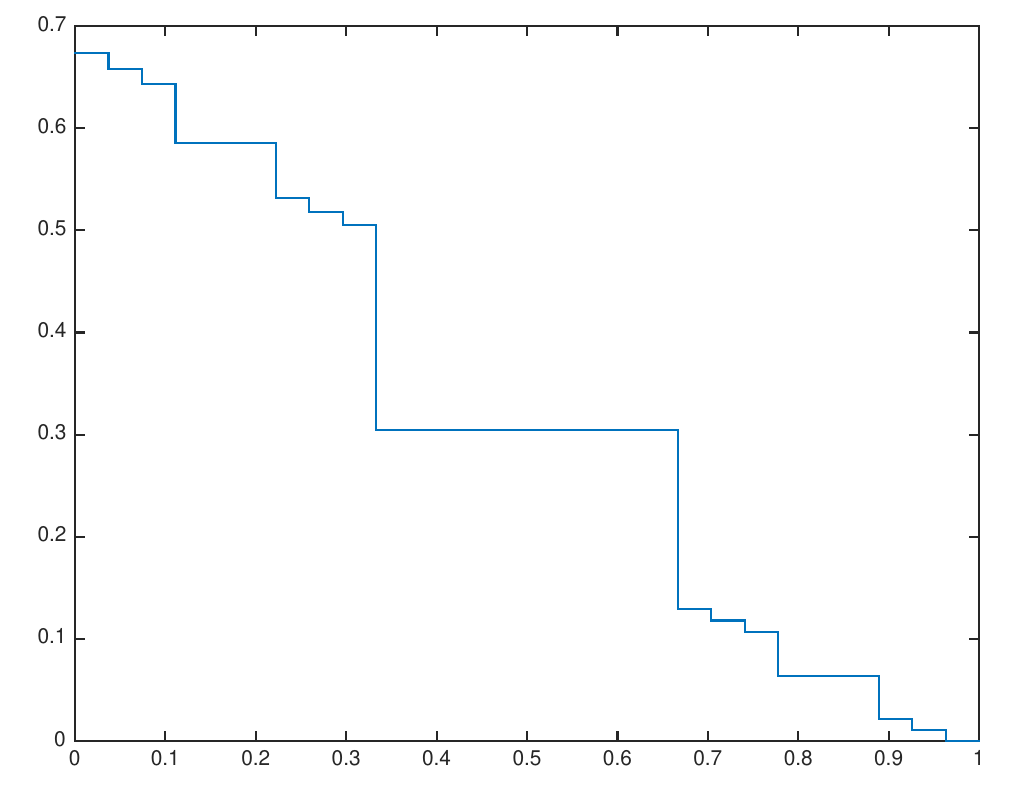} } }
        \end{subfigure}  
 \\ 
          \begin{subfigure}[Solution $p^{9}$, Stage $n = 9$.]
                { 
                {\includegraphics[scale = 0.45]{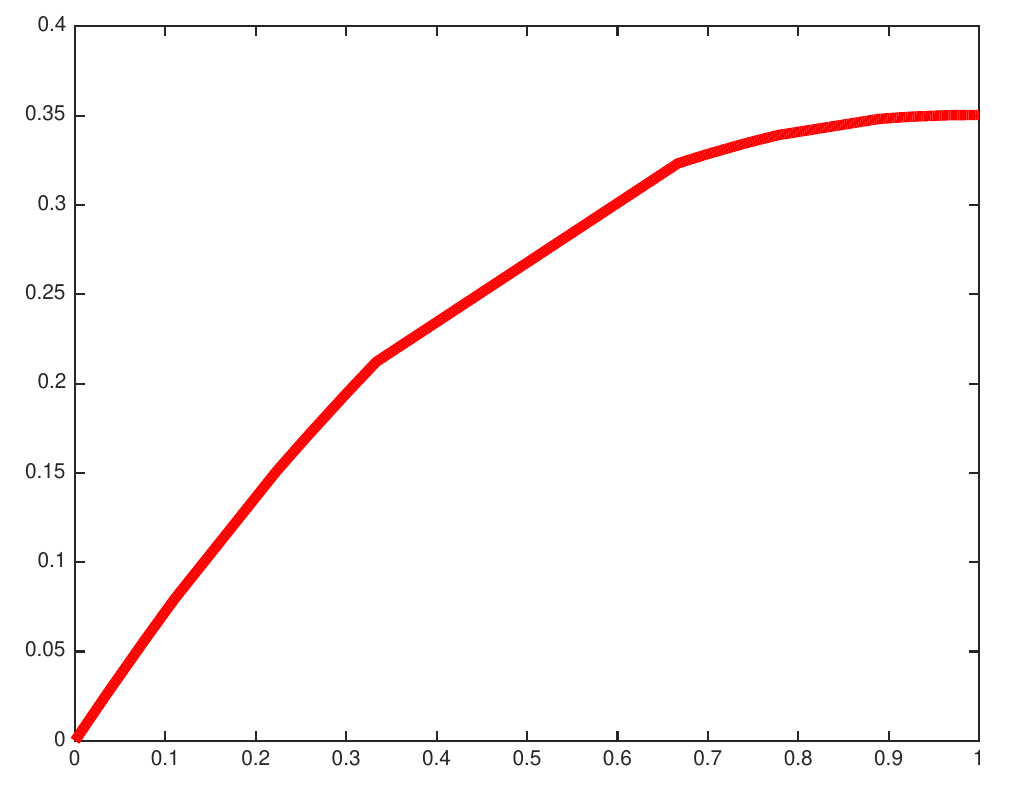} } }                
        \end{subfigure} 
%
        ~ 
          \begin{subfigure}[Derivative $\del p^{9}$, Stage $n = 9$.]
                { 
{\includegraphics[scale = 0.44]{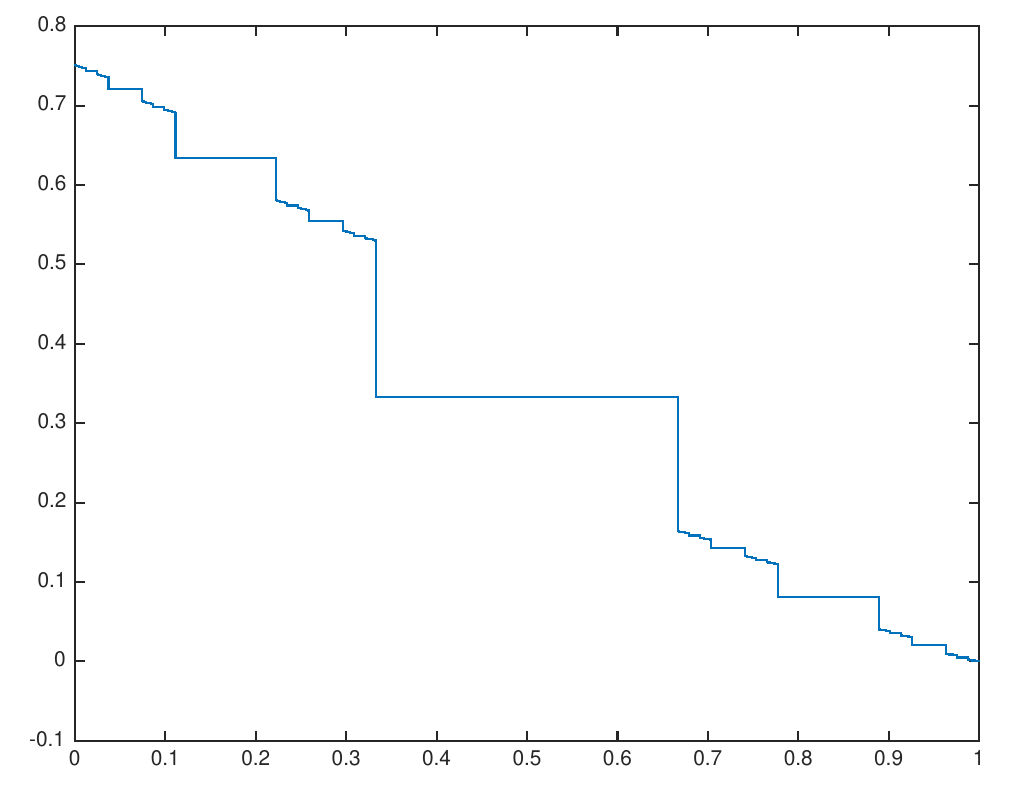} } }                
        \end{subfigure} 
\caption{\textsc{Solutions Example} \ref{Ex Scaled Example}. Storage coefficient $\beta$ defined in \textsc{Equation} \eqref{Eq Example scaled storage coefficient}, forcing term $F = 0$ and interface forcing term $f$, defined in \eqref{Def Deterministic Interface Forcing Term}. The functions depicted in figures (a) and (c) are the corresponding solutions when the microstructures $\B_{4}$ and $\B_{9}$ are present, see \eqref{Def Partial Triadic Extremes Set}. Figures (b) and (d) are the respective derivatives. The vertical lines in derivatives' graphics are included for optical purposes only.}\label{Fig Solutions Ex2} 
\end{figure}
\end{example}
\begin{example}[Random Behavior Example]\label{Ex Random Behavior Example}
This third and last example is a probabilistic variation of \textsc{Example} \ref{Ex Basic Example}. We simply introduce uncertainty in the storage coefficient $\beta$ by considering it a random variable, uniformly distributed on an interval i.e., 

\begin{align}\label{Def Probabilistic Experimental Storage Coefficient}
& \beta: \B\rightarrow \Big[\frac{3}{4}\, , \frac{5}{4} \Big] \, , &
&  \beta\sim \text{uniformly} .
\end{align} 
Due to the Law of Large Numbers, the average of $\beta$ is the constant function $\bar{\beta} = 1$. Therefore, due to the linearity of the differential equation, the ``averaged" function $\bar{p}$ of the solutions corresponding to the realizations is precisely the solution of \textsc{Example} \ref{Ex Basic Example}. Below, \textsc{Table} \ref{Tbl Convergence Ex3} presents the results for eight numerical experiments, each of them consisting in averaging the outcome of twenty random realizations; we denoted this average by $\bar{p}_{n}$. The experiments were executed for two different stages, namely $n = 4$ and $n = 8$, consequently finite versions of the storage coefficient $\beta_{i} \defining \beta\big\vert_{\B_{i}}$ (introduced in\eqref{Def Probabilistic Experimental Storage Coefficient}) were used, namely

\begin{align}\label{Def Finite Probabilistic Experimental Storage Coefficient}
&  \beta_{i}: \B_{i}\rightarrow \Big[\frac{3}{4}\, , \frac{5}{4} \Big] \, , &
&  \beta_{i}\sim \text{uniformly} \, , &
&  i = 4, 8 .
\end{align} 
%
%
%
%
%
\begin{table*}[h!]
\def\arraystretch{1.4}
\begin{center}
\rowcolors{2}{gray!25}{white}
 \begin{tabular}{c c c c c c}
\hline
\rowcolor{gray!50}
 & \multicolumn{2}{ c }{Stage $n = 4$} & \multicolumn{2}{ c }{Stage $n = 8$ }\\
\rowcolor{gray!50}
Experiment
& $\Vert  \, \bar{p}_{n}- \bar{p}  \, \Vert_{L^{2}(0,1)} $ 
& $\Vert  \, \bar{p}_{n}- \bar{p}  \,  \Vert_{H^{1}(0,1)} $ 
& $\Vert  \, \bar{p}_{n}- \bar{p}  \,  \Vert_{L^{2}(0,1)} $ 
&  $\Vert  \, \bar{p}_{n}- \bar{p}  \,  \Vert_{H^{1}(0,1)} $ 
\\
\hline
1
& 0.000119184000335  & 0.001061092385691
& 0.000003969114134 &  0.000255622745433
\\
2
& 0.000212826056946 &  0.001524984778876
& 0.000011927845271 &  0.000729152212225
\\
3 
& 0.000251807786501 &  0.002145242883204
& 0.000006707500555 &  0.000433471661184
\\
4
& 0.000118807313682  &  0.001071251072221
& 0.000006848283589 &  0.000399676826771
\\
5 
& 0.000200559658088  &  0.001539679789493
& 0.000003353682741 &  0.000236079667248
\\
6
& 0.000274728964017  & 0.001805225518962
& 0.000017705174519  & 0.001055292921996 
\\
7
& 0.000262099625343 &  0.001921561839120
& 0.000013849036017 &  0.000696257848189
\\
8
& 0.000338584909360 &  0.001974469520292
& 0.000004372910306 &  0.000289302820258
%
\\
\hline
Average
& 0.000250408088105 &  0.001817080230904
& 0.000008011140143  & 0.000473741408661
\\
Variance
& $0.005804534\times 10^{-6}$  & $0.164983904 \times 10^{-6} $
&  $0.000027860\times 10^{-6}$  & $0.083787371 \times 10^{-6}$
\\
\hline
\end{tabular}
\caption{\textsc{Example} \ref{Ex Random Behavior Example}: Experiments, Each Consisting of 20 Random Realizations.}\label{Tbl Convergence Ex3}
\end{center}
\end{table*}
In the table \ref{Tbl Convergence Ex3} below, the $L^{2}$ and $H^{1}$-norms for the difference of the partial averages $\bar{p}_{n}$ and the Ces\`aro average are displayed. Clearly, convergence is observed in both cases and, due to the Law of Large Numbers, better behavior is observed at the stage $n = 8$ over the stage $n=4$, as expected. The average and variance are also presented at the bottom of \textsc{Table} \ref{Tbl Convergence Ex3}, which shows better behavior for the more developed stage $n=8$ over the stage $n=4$, not only in terms of centrality, but also in terms of deviation. 

Finally, the solutions for some realizations are presented in \textsc{Figure} \ref{Fig Solutions Ex3}, where the microstructure is $\B_{6}$. Clearly, the random realizations' solutions are similarly shaped to its average value as expected.
\begin{figure}[h!] 
\vspace{-1.3cm}
        \centering
        \begin{subfigure}[Realization 1: Solution $p^{6}$.]
                { 
                {\includegraphics[scale = 0.45]{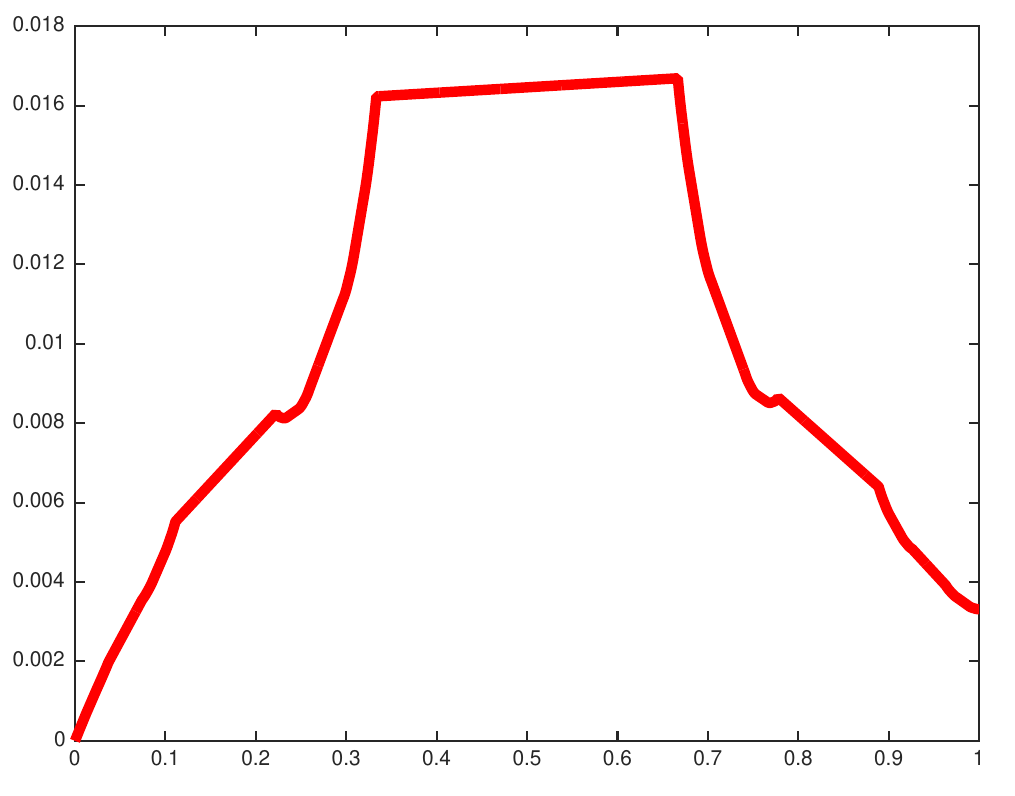} } }
        \end{subfigure}
%
          \begin{subfigure}[Realization 1: Derivative of the Solution $\del p^{6}$.]
                { 
                {\includegraphics[scale = 0.44]{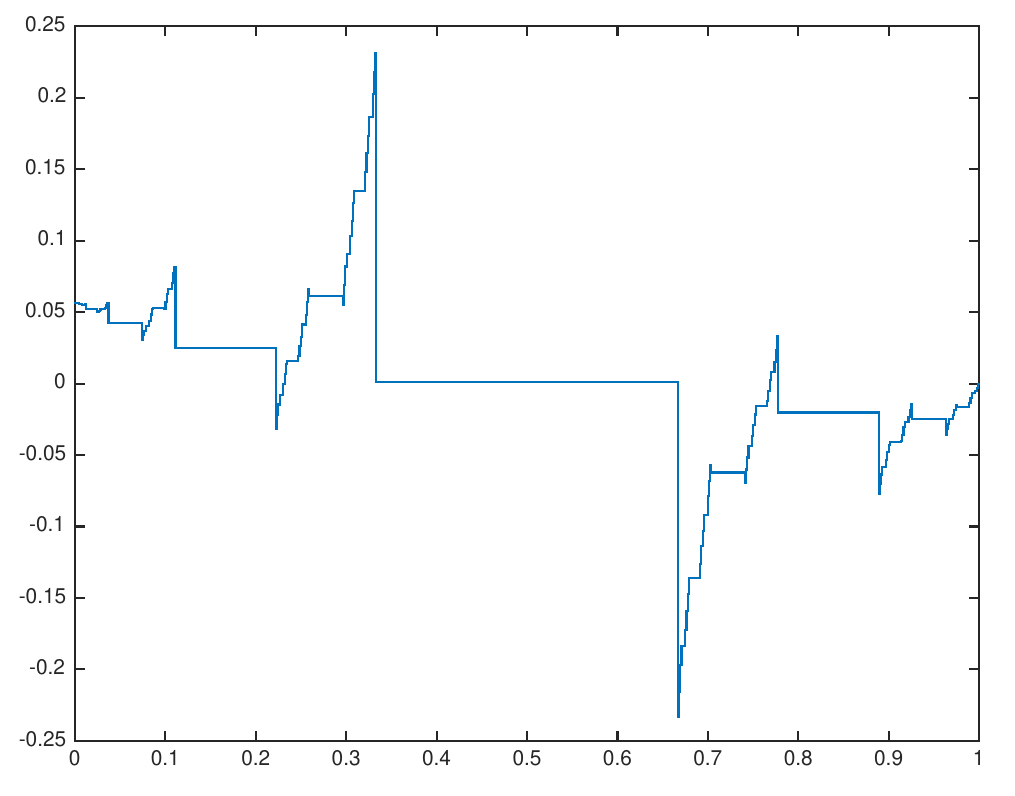} } }                
        \end{subfigure} 
        \\
        \begin{subfigure}[Realization 2: Solution $p^{6}$.]
                { 
{\includegraphics[scale = 0.45]{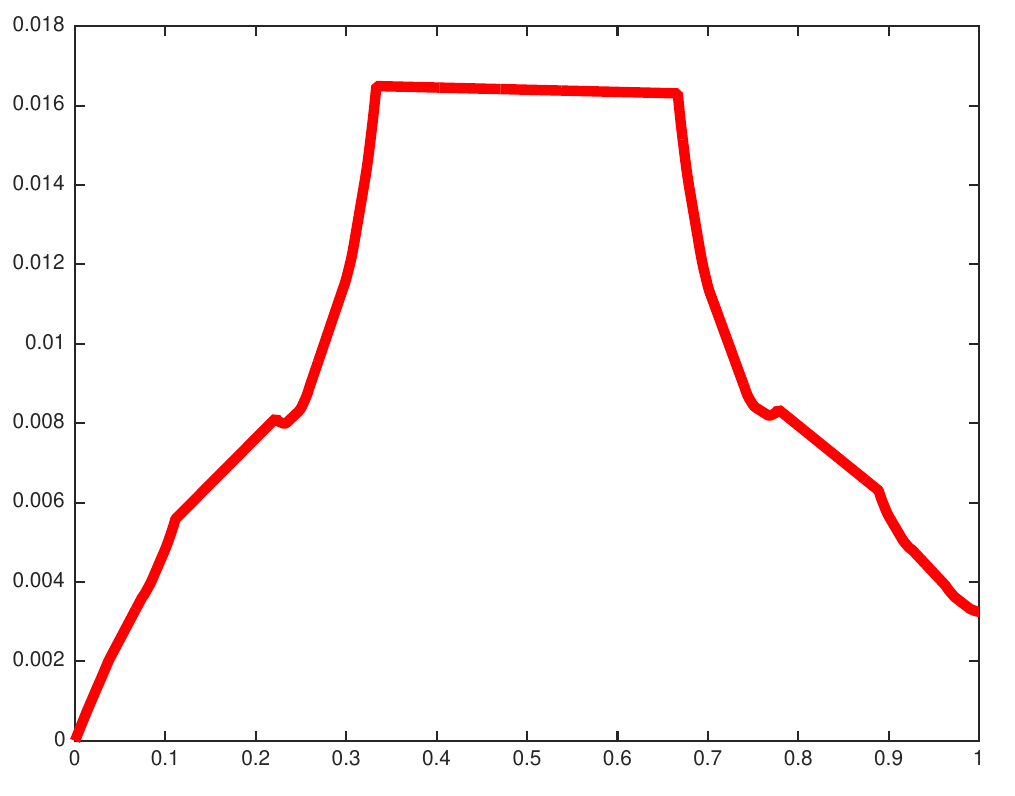} } }
        \end{subfigure} 
%
          \begin{subfigure}[Realization 2: Derivative of the Solution $\del p^{6}$.]
                { 
{\includegraphics[scale = 0.44]{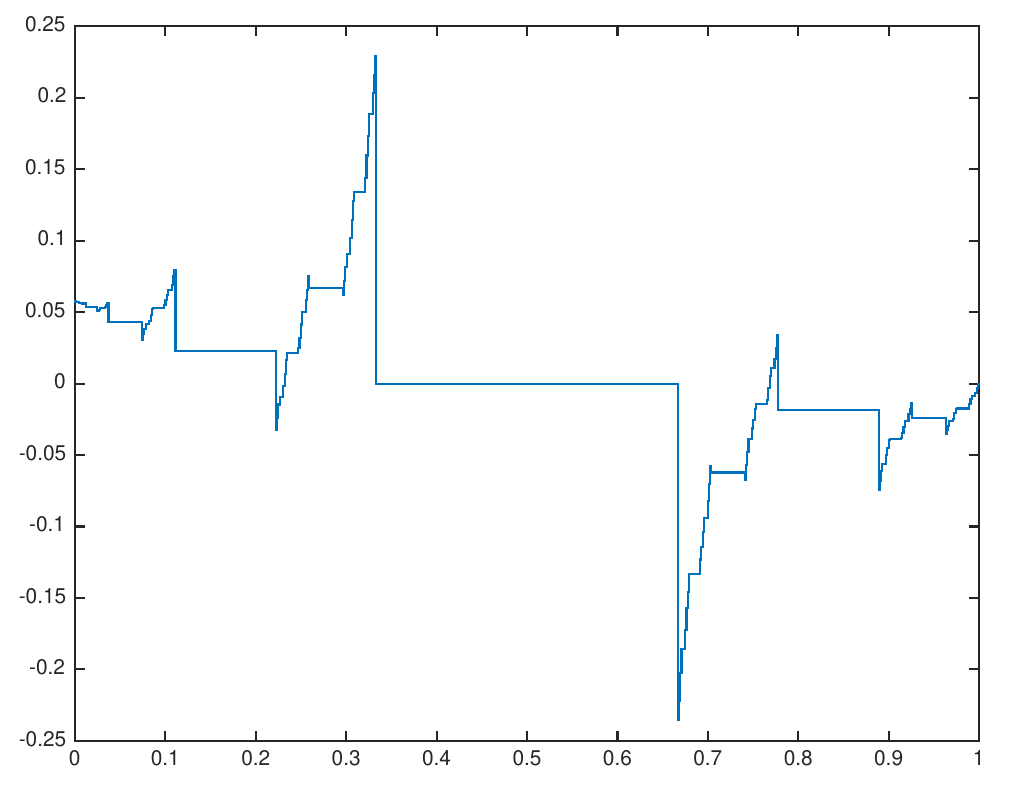} } }   
        \end{subfigure} 
        \\
        \begin{subfigure}[Realization 3: Solution $p^{6}$.]
                { 
{\includegraphics[scale = 0.45]{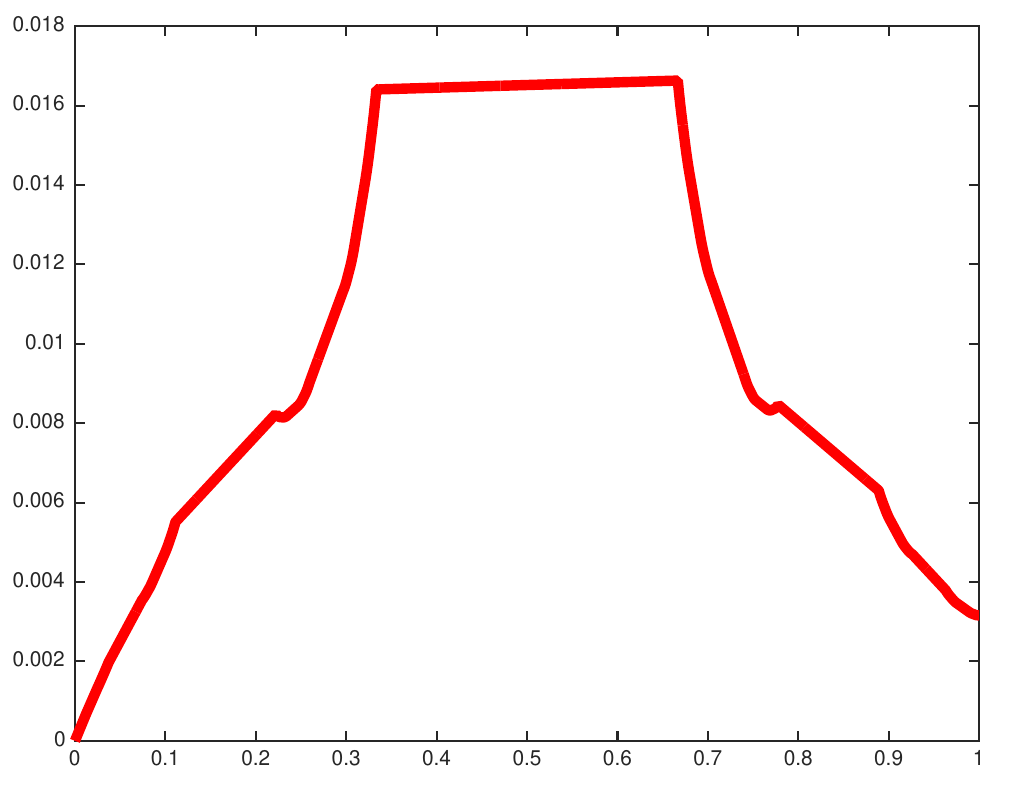} } }
        \end{subfigure} 
%
          \begin{subfigure}[Realization 3: Derivative of the Solution $\del p^{6}$.]
                { 
{\includegraphics[scale = 0.44]{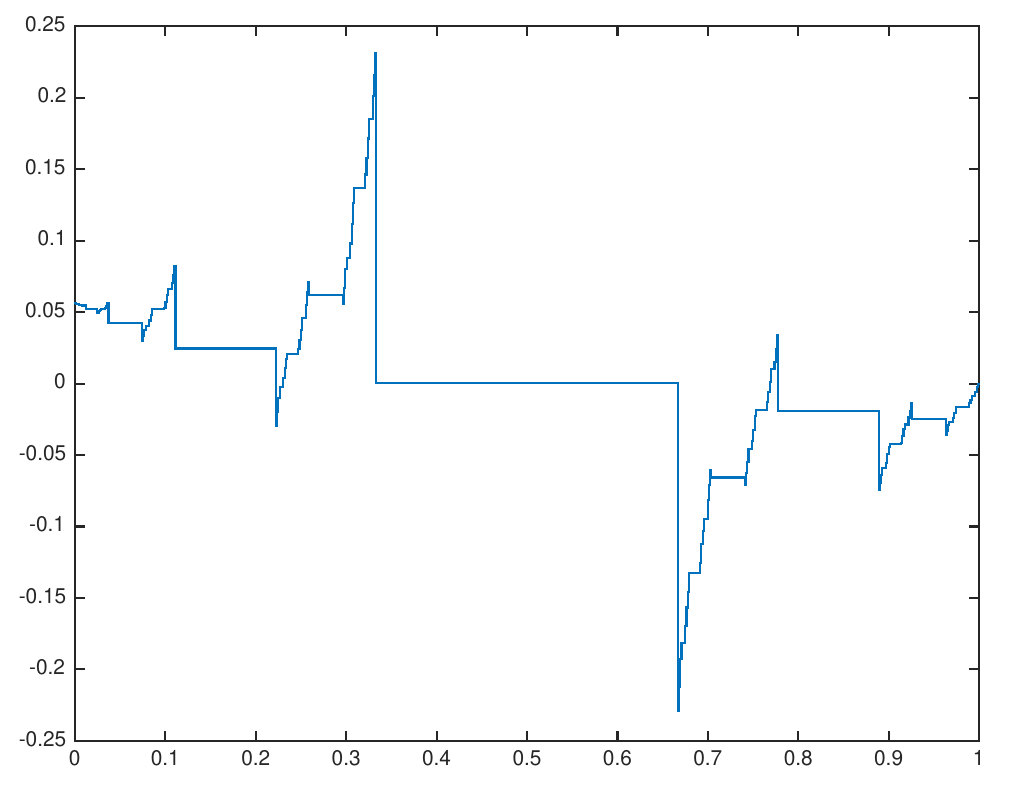} } }               
        \end{subfigure} 
\caption{\textsc{Random Realizations Example} \ref{Ex Random Behavior Example}. Storage coefficient $\beta$ defined in \textsc{Equation}\eqref{Eq Example scaled storage coefficient}, forcing term $F = 0$ and interface forcing term $f$ defined in \eqref{Def Deterministic Interface Forcing Term}. For three random realizations, the functions depicted in figures (a), (c) and (e) are the solutions and figures (b), (d) and (f) are the corresponding derivatives. The microstructure is $\B_{6}$, see \textsc{Equation} \eqref{Def Partial Triadic Extremes Set}. The vertical lines in the derivatives' graphics are included for optical purposes only.}\label{Fig Solutions Ex3}
\end{figure}
\end{example}
%
%
%
%
%
\subsection{Closing Observations}\label{Sec Closing Observations Numerical}
%
%
%
%
%
%
\begin{enumerate}[(i)]
\item 
The Authors tried to find experimentally, a rate of convergence of the type $\Vert p^{n+1}_{i} - p^{n}_{i} \Vert = \mathcal{O}\big(\Vert p^{n}_{i} - p^{n-1}_{i} \Vert^{t_{i}} \big)$, $i = 1, 2$, using the well-know estimate 

\begin{align*}
& t_{i} \sim \frac{\log \Vert  p^{ \, n+1}_{i}- p^{ \, n }_{i}  \Vert - 
\log   \Vert  p^{ \, n}_{i} - p^{ \, n-1}_{i} \Vert }{
\log   \Vert  p^{ \, n}_{i} - p^{ \, n-1}_{i} \Vert
- \log \Vert  p^{ \, n-1}_{i}- p^{ \, n -2}_{i}  \Vert } \, ,&
& i = 1, 2.
\end{align*}
The sampling was made on the sequence of stages $3,4,5,6,7,8,9$; experiments were run for \textsc{Examples} \ref{Ex Basic Example} and \ref{Ex Scaled Example}. However, in both cases, no solid numerical evidence suggesting an order of convergence for the phenomenon was detected. 

\item Additional experiments for the \textsc{Random Behavior Example} \ref{Ex Random Behavior Example} were executed. The probabilistic variations involved were

\begin{enumerate}[(a)]
\item 
Experiments for the stages $n = 3, 5, 6, 7$ and $9$.

\item Experiments with different number of realizations, namely $20$, $60$ and $80$, depending on the computational time demanded.

\item Using a storage coefficient $\beta: \B\rightarrow I$, with different range intervals $I$ as well as uniform and normal distributions for $\beta$.

\item Probabilistic variations of the forcing terms.

\item Combinations of one or more of the previous factors.

\item Execution of the aforementioned probabilistic variations adjusted to the scaled \textsc{Example} \ref{Ex Scaled Example}.   

\end{enumerate} 
In all the cases, convergence behavior was observed as expected. Naturally, the quality of convergence deteriorates depending on the deviation of the distributions as well as the combination of uncertainty factors introduced in the experiment. 

\item All of the previously mentioned scenarios were also executed in the same code for the case of homogeneous Dirichlet boundary conditions on both ends i.e., $\pn(0) = \pn(1) = 0$. As expected convergence behavior is observed, comparable with the corresponding analogous version used along the analysis of this paper i.e., Dirichlet-Neumann boundary conditions $\pn(0) = 0$, $\del \pn(1) = 0$.  

\end{enumerate}
%
%
%
%
%
%
%
%
%
\section{Conclusions and Final Discussion}
%
%
The present work yields several accomplishments as well as limitations listed below.
\begin{enumerate}[(i)]
\item The unscaled storage model presented in \textsc{Section} \ref{Sec A sigma finite Interface Microstructure} is in general not adequate, as it excludes most of the important cases of fractals.

\item The scaled storage model presented in \textsc{Section} \ref{Dec Fractal Scaling Modeling} is suitable for an important number of fractal microstructures. However, the asymptotic variational model \eqref{Pblm Variational Microstructure Good Coefficients} is equivalent to a pointwise strong model \eqref{Pblm Strong Microstructure Good Coefficients}, only if the closure of the microstructure $\cl(\B)$ is negligible i.e., if it has null Lebesgue measure. Such hypothesis excludes important cases, e.g. the family of ``fat" Cantor sets in 1-D or the ``fat" versions of the classic fractal structures in 2-D and/or 3-D (see \cite{Falconer}). 
In these cases, only the ``averaged normal flux" \textsc{Statement} \eqref{Eq strong limit interface accumulation points extra hypothesis} can be concluded for the accumulation points of the microstructure $\B'$.

\item In order to overcome the deficiency previously mentioned, a first approach would be to take the traditional treatment of solving strong forms in fractal domains (as in \cite{Strichartz3, Strichartz1}) and then, try to ``blend" it with the point of view presented here. Such analysis is to be pursued in future research.  

\item The fractal microstructures addressed in this work are self-similar. This requirement is important only for the second model (\textsc{Section} \ref{Dec Fractal Scaling Modeling}) because it scales the storage coefficient $\beta$ (\textsc{Equation} \eqref{Eq scaled storage coefficient}), according to the geometric detail of the structure at every level. In particular,  in order to scale the storage coefficient adequately, accurate estimates of the growth rate of the microstructure $\B$ from one level to the next are needed. 

\item The self-similarity requirement for the microstructure $\B$ in the introduction of the scaled model excludes the important family of the self-affine fractals; this type of microstructures is a topic for future work. On the other hand, the unscaled model of \textsc{Section} \ref{Sec A sigma finite Interface Microstructure} does not require such detailed knowledge of the microstructure because it avoids scaling. Although it is likely to be unsuited for the analysis of self-affine microstructures it may be a good starting point, when looking for the needs to model this case.    

\item It is important to observe the relevance of self-similarity versus the fractal dimension in the scaled model. While the self-similarity of the microstructure is the corner stone of the storage coefficient scaling, introduced in \textsc{Equation} \eqref{Eq scaled storage coefficient} (hence, it can not be given up), the model is more flexible with respect to the fractal dimension of $\B$, as long as it is not the same of the ``host" domain $\Omega$.

\item We stress that the input needed by the present result, in terms of geometric information on the fractal structure, does not have to be as detailed as in the strong forms PDE analysis on fractals. 

\item The random experiments presented in \textsc{Example} \ref{Ex Random Behavior Example}, as well as those only mentioned in \textsc{Subsection} \ref{Sec Closing Observations Numerical}, furnish solid numerical evidence of good behavior for probabilistic versions of the unscaled and the scaled models respectively. Additionally, it is important to handle certain level of uncertainty because, having a deterministic description of the fractal microstructure is a very strong hypothesis to be applicable in realistic scenarios. In the Authors' opinion, this is justification enough to pursue rigorous analysis of these problems, which will be addressed in future work.

\item In \textsc{Example} \ref{Ex Random Behavior Example}, uncertainty was introduced in the storage coefficient $\beta$, or the forcing term $f$, however the geometry of the microstructure was never randomized. In several works (e.g. \cite{GrafS, HutchinsonRuschendorf}) the self-similarity is replaced by the concept of statistical self-similarity, in the sense that scaling of small parts have the same statistical distribution as the whole set. Clearly, this random property is consistent with the scaling of $\beta$ \eqref{Eq scaled storage coefficient} in \textsc{Definition} \ref{Def scaled storage coefficient}. Consequently, the statistical self-similarity is a future line of research, in order to address geometric uncertainty of the fractal microstructure $\B$. In particular, the fractal percolation microstructures, are of special interest for real world applications, see \cite{ChayesDurret, DekkingMeester}.

\item The execution of all the numerical experiments shows that the code becomes unstable beyond the 9th stage of the Cantor set construction. This suggests that in order to overcome these limitations, an adaptation of the FEM method has to be developed, targeted to the microstructure of interest. This aspect is to be analyzed in future research.

\item 
The study of fractal microstructures in 2-D and 3-D are necessary for practical applications and in higher dimensions for theoretical purposes. Since passing from one dimension to two or more dimensions, increases significantly the level of complexity in the microstructure and in the equation, considerable challenges are to be expected in this future research line. 

\item Another important analysis to be developed is the study of the models both, scaled and unscaled, in the mixed-mixed variational formulation introduced in \cite{MoralesShow2}. On one hand, this approach allows great flexibility for the underlying spaces of velocity and pressure, which can constitute an advantage with respect to the treatment presented here. On the other hand, the mixed-mixed variational formulation allows modeling fluid exchange conditions across the interface of greater generality than the conditions in \eqref{Eq strong n-stage interface}, used in the present work; this advantage can contribute significantly to the development of the field.

\end{enumerate}
%
%


%
%
\section*{Acknowledgements}
The Authors wish to acknowledge Universidad Nacional de Colombia, Sede Medell\'in for its support in this work through the project HERMES 27798. The authors also wish to thank Professor Ma\l{}gorzata Peszy\'nska from Oregon State University, for authorizing the use of code \textbf{fem1d.m} \cite{PeszynskaFEM} in the implementation of the numerical experiments presented in \textsc{Section} \ref{Sec Numerical Experiments}. It is a tool of remarkable quality, efficiency and versatility, which has contributed significantly to this work.
%
%
%
%
%
%
%

%
%
\end{document}